\documentclass[reqno,11pt]{amsart} 
\usepackage{esint,mathrsfs,amsmath,amssymb,amsfonts,bbm}
\usepackage{verbatim,pifont}

\usepackage{color}
\usepackage[hyperfootnotes=false]{hyperref}
\hypersetup{
  colorlinks,
  citecolor=blue,
  linkcolor=red,
  urlcolor=blue}

\usepackage[left=1.4in, right=1.4in, bottom=1.4in]{geometry}
\newtheorem{theorem}{Theorem}[section]

\newtheorem{corollary}[theorem]{Corollary}
\theoremstyle{definition}

\theoremstyle{remark}
\newtheorem{remark}[theorem]{Remark}

\numberwithin{equation}{section}

\newcommand{\Z}{\mathbb{Z}}

\newcommand{\R}{\mathbb{R}}
\newcommand{\T}{\mathbb{T}}
\newcommand{\e}{\varepsilon}

\newcommand{\pa}{\partial}

\newcommand{\Lm}{\Lambda}

\newcommand{\na}{\nabla}

\newcommand{\txtdiv}{\text{\rm div}}

\newcommand{\sgn}{\text{sgn}}

\begin{document}

\title[Uniform Calder\'{o}n-Zygmund estimates]{Uniform Calder\'{o}n-Zygmund estimates\\ in multiscale elliptic homogenization}


\author{Weisheng Niu}
\address{W. Niu: School of Mathematical Science, Anhui University,
Hefei,  China}
\email{niuwsh@ahu.edu.cn}

\author{Jinping Zhuge}
\address{J. Zhuge: Morningside Center of Mathematics, Academy of Mathematics and systems science,
Chinese Academy of Sciences, Beijing, China.}
\email{jpzhuge@amss.ac.cn}


\subjclass[2020]{35B27}


\begin{abstract}
    This paper is concerned with the elliptic equation $-\txtdiv (A_\e \na  u_\e) = \txtdiv f$ in a bounded $C^1$ domain, where $A_\e$ takes a form of $A_\e(x) = A(x/\e_1, x/\e_2,\cdots, x/\e_n)$, with $A(y_1,y_2,\cdots,y_n)$ being 1-periodic in each $y_i$. We prove the uniform Calder\'{o}n-Zygmund estimate, namely, the uniform $L^p$ boundedness of the linear map $f\mapsto \nabla u_\e$ for any $p\in (1,\infty)$ with a constant independent of small parameters $(\e_1,\e_2,\cdots, \e_n) \in (0,1]^n$. Our result includes the uniform Calder\'{o}n-Zygmund estimate in quasiperiodic elliptic homogenization (even without the Diophantine condition), which was previously unknown. The proof novelly combines the Dirichlet's theorem on the simultaneous Diophantine approximation from number theory, a technique of reperiodization, reiterated periodic homogenization and a large-scale real-variable argument. Using the idea of reperiodization, we also obtain some large-scale or mesoscopic-scale Lipschitz estimates.

    \noindent\textbf{Keywords:} Multiscale homogenization, Diophantine approximation, Calder\'{o}n-Zygmund estimate.
\end{abstract}

\maketitle
\tableofcontents

\section{Introduction}


In this paper, we study the uniform regularity for a family of linear elliptic equations or systems in divergence form in a bounded domain $\Omega \subset \R^d$,
\begin{align}\label{operaotr-n}
-\txtdiv ( A_\e(x) \na u_\e) = \txtdiv f,
\end{align}
where the coefficient matrix takes a form of $A_\e(x) = A(x/\e_1, x/\e_2,\cdots, x/\e_n)$, and $(\e_i)_{1\le i\le n} \in (0,1]^n$ are small parameters describing $n$ different oscillating scales of the coefficients. We assume that
 $A(y_1,y_2,\cdots, y_n):  \R^{d\times n} \mapsto \R^{d\times d}$ is 1-periodic in each $y_i \in \R^d$ with $1\le i\le n$. Due to the symmetry of $y_i$'s, it is harmless to assume $\e_1\ge \e_2 \ge \cdots \ge \e_n>0$. The full assumptions on $A$ will be given shortly.

The homogenization theorem (or quantitative convergence rates) and uniform regularity are the fundamental questions in homogenization theory. These two questions have been well answered for the equation \eqref{operaotr-n} in the case of one oscillating scale (i.e., $n=1$); see recent monographs \cite{shenbook2018,armstrongbook2} and references therein. For the case of multiple oscillating scales, we refer to our recent work \cite{niuzhuge2023} for a brief survey. It is important to point out that before the work \cite{niuzhuge2023}, the homogenization theorem and uniform regularity were considered closely interconnected and always appear together, in the sense that the uniform regularity can be derived as a consequence of convergence rates either by a compactness method \cite{al87,klsjams13} or a quantitative method (excess decay iteration \cite{gloria2014regularity,armstrongan2016,shenapde2017} or a real-variable argument \cite{caffarelli1998}). For the equation \eqref{operaotr-n} with $n\ge 2$, we have the qualitative homogenization theorem or a quantitative convergence rate to a deterministic homogenized equation only if the scales $(\e_1,\e_2,\cdots, \e_n)$ are separated or well-separated \cite{allaire1996}, respectively. Recall that the scales $(\e_1,\e_2,\cdots, \e_n)\in (0,1]^n$ are well-separated if $\e_{i+1} \lesssim \e_i^{\alpha_i}$ for some $\alpha_i > 1$.
Consequently, under this scale-separation condition, 
the uniform regularity estimates for \eqref{operaotr-n} were established in \cite{nsxjfa2020} by a quantitative method. Yet, as pointed out by Avellaneda in \cite{avellaneda1996} (see also \cite{geradcmap2003}) such an assumption is not adequate for treating the most general problems of transport and diffusion in self-similar (particularly random) media. 
The main contribution of \cite{niuzhuge2023} is that, without any separation condition on $(\e_i)_{1\le i\le n}$, we obtained the uniform $C^\alpha$ estimate for any $\alpha \in (0,1)$ by the compactness method combined with a novel scale-reduction process. This result was unexpected since it shows for the first time that the uniform regularity holds stably even without homogenization, and in some sense the averaging effect always takes places across all the mesoscopic scales.

This paper is a continuation of our previous work \cite{niuzhuge2023}. It was conjectured in \cite{niuzhuge2023} that with $f = 0$ in \eqref{operaotr-n} and without any scale-separation condition, the solution admits uniform Lipschitz estimate. We should mention that the gradient estimate (like Lipschitz estimate or Calder\'{o}n-Zygmund estimate) cannot be established by the soft compactness method as in our previous paper without the help of correctors or convergence rates. In this paper, we partially solve the conjecture by proving a slightly weaker result, namely, the gradient $\nabla u_\e$ is uniformly bounded in $L^p$ spaces for any $p<\infty$. In particular, this recovers the uniform $C^\alpha$ estimate in \cite{niuzhuge2023} in view of the Sobolev embedding theorem. Moreover, our proof is quantitative (in contrast to the compactness method in \cite{niuzhuge2023}) in the sense that the constant can be computed explicitly.

Before stating the main result, we list the assumptions on the coefficient matrix $A$:
\begin{itemize}
	\item Strong ellipticity condition: there exists $\Lambda \in (0,1]$ so that
	\begin{equation}\label{ellipticity}
		\Lambda |\xi|^2 \le \xi \cdot A(y_1,\cdots,y_n)\xi \quad  \text{and} \quad |A(y_1,\cdots,y_n)\xi| \le \Lambda^{-1} |\xi|, 
	\end{equation}
    for every $\xi \in \R^d$ and   for any $(y_1,\cdots,y_n)\in  \R^{d\times n}.$

	\item Periodicity: for any $(y_1,\cdots,y_n) \in \R^{d\times n}$ and $(z_1,\cdots,z_n)\in \Z^{d\times n}$,
	\begin{equation}\label{periodicity}
		A(y_1+z_1, \cdots, y_n+z_n) = A(y_1,\cdots,y_n).
	\end{equation}
	
	\item Smoothness: there exists $\tau\in (0,1]$ and $L>0$ such that
	\begin{equation}\label{smoothness}
		|A(y_1,\cdots,y_n) - A(y_1',\cdots,y_n')| \le L \big\{|y_1'-y_1|+\cdots +|y_n'-y_n|\big\}^\tau
	\end{equation} for any   $(y_1,\cdots,y_n), (y_1',\cdots,y_n')\in \R^{d\times n}$.
\end{itemize}

We now state our main result.
 \begin{theorem}\label{th-w1p}
 Let $\Omega$ be a bounded $C^1$ domain. Assume $A$ satisfies \eqref{ellipticity}, \eqref{periodicity} and \eqref{smoothness} and $(\e_i)_{1\le i\le n} \in (0,1]^n$. Assume $1<p<\infty$. Let $u_\e$ be the weak solution to the Dirichlet problem
 \begin{align} \label{eq.Dirichlet}
 - \txtdiv \big( A_\e  \na u_\e \big)= \txtdiv f \, \text{ in }   \Omega, \quad   u_\e=0 \,\text{ on }  \pa \Omega,
 \end{align}
where $ A_\e=A(x/\e_1, x/\e_2, \cdots, x/\e_n)$ and  $f\in L^p(\Omega)^d$. Then $\na u_\e \in L^p(\Omega)^d$ and
 \begin{align} \label{th-b-w1p-re}
\| \nabla u_\e \|_{L^p(\Omega)} \le C\| f\|_{L^p(\Omega)},
\end{align}
where $C$ depends only on $d, n, \Lambda, p$, $(\tau,L)$ in \eqref {smoothness}, and $\Omega$.
 \end{theorem}

The same uniform estimate holds if the Dirichlet boundary condition in \eqref{eq.Dirichlet} is replaced by the compatible Neumann boundary condition
\begin{equation}
    \nu \cdot A_\e \na u_\e= -\nu \cdot f \,\text{ on }  \pa \Omega,
\end{equation}
where $\nu$ is the unit outer normal vector of $\partial \Omega$.

The new ingredient in the proof of Theorem \ref{th-w1p} is the Dirichlet's theorem (see Theorem \ref{th.Dirichlet}) on simultaneous Diophantine approximation from number theory, which enters into a new technique of \emph{reperiodization} and allow us to quantitatively separate one scale from the rest $n-1$ scales, leading to a quantitative gradient approximation (this is not achievable by the rough scale separation argument in \cite{niuzhuge2023}). Precisely, for a given sequence $(\e_1,\e_2,\cdots,\e_n)$ in nonincreasing order and any number $Q>1$, we can find a new 1-periodic matrix $A^\sharp$, depending on $Q$ and the ratios $\e_{n}/\e_i$, such that $A_\e(x)$ can be rewritten as
\begin{equation}
    A(x/\e_1,x/\e_2,\cdots, x/\e_n) = A^\sharp(x/\e'_1,x/\e'_2,\cdots, x/\e'_n),
\end{equation}
with some $(\e_1',\e_2',\cdots, \e_n') \in (0,\infty)^n$ and the last scale $\e_n'$ is \emph{$Q$-separated} from the rest $\e_i'$'s (i.e., $\e'_i \ge Q\e_n'$ for all $1\le i\le n-1$); see Section \ref{sec3.1} for details. Theoretically, $Q$ can be chosen as large as $(\e_n/r)^{-\sigma}$ with some $\sigma>0$, yielding a well-separation condition and thus a good convergence rate to an approximate problem in $B_r$ with at most $n-1$ scales (see Remark \ref{rmk.rate}). However, as a payoff 
the regularity of $A^\sharp(y_1,y_2,\cdots, y_n)$ in $y_n$ will become worse as $Q$ increases, which leads to a small-scale estimate depending on $Q$ in a blow up argument. 
Fortunately, a sufficiently large $Q$ independent of $(\e_i)_{1\le i\le n}$ will be just enough for our proof. With the $Q$-separation condition, we can perform a quantitative argument of the reiterated homogenization and approximate the solution of the original equation with $n$ oscillating scales by a solution of a new equation with at most $n-1$ oscillating scales. Hence, an inductive argument on the number of scales combined with a careful real-variable argument (see Theorem \ref{th W1P real}, which is a quantitative refined version of the one originating in \cite{caffarelli1998}) involving a \emph{double-averaging estimate} will complete the proof.

It is crucial to point out that the proofs for $C^\alpha$ estimate (a type of Schauder estimates) and $L^p$ gradient estimate (Calder\'{o}n-Zygmund estimate related to singular integrals) are essentially different, for the latter is stronger and more informative. As in \cite{niuzhuge2023}, the $C^\alpha$ estimate can be derived with a qualitative $H$-convergence theorem and a compactness method (the correctors are not needed; also see \cite[Theorem 3.1] {shenapde2015} or \cite[Theorem 6.1]{shen-zhuge2016} for similar situations). However, the estimate of gradient must involve a quantitative convergence rate or the fine properties of correctors. While our proof of Theorem \ref{th-w1p} also uses a scale-reduction theorem and an inductive argument on the number of scales, it is exactly the Dirichlet's theorem that helps us quantify the error in the scale-reduction process and provides much stronger implications and more information about the  behavior  of the solutions. In particular, it possibly explains the reason behind the occurrence of the mesoscopic averaging effect in arbitrary multiscale media.

As a direct consequence, Theorem \ref{th-w1p} implies the uniform Calder\'{o}n-Zygmund estimate for the operator $-\txtdiv ( A (x/\e) \na )$ with arbitrary quasiperiodic coefficient $A$ (which was previously known only under a Diophantine condition). We recall that $A(y)$ is said to be quasiperiodic if $A(y)= B( M y )$, where $B(w)$ is $1$-periodic in $w \in \R^{N}$ with $1\le N \in \mathbb{N}$ and $ M $ is an $N\times d$ constant real matrix. Note that we do not have any restriction on the entries of $M$ and therefore $A$ is allowed to be periodic with a very degenerate periodic cell.

 
\begin{corollary}\label{coro1.3}
Let $\Omega$ be a bounded $C^1$ domain. Assume that $A(y) = B(My)$ is quasi-periodic as above and $B$ satisfies \eqref{ellipticity}, \eqref{periodicity} and \eqref{smoothness} (with $n=1$).  Let $1<p<\infty$, $f\in L^p(\Omega)^d$, and $u_\e$ be the weak solution to
 \begin{align*}  
 -\txtdiv \big( A(x/\e)  \na u_\e \big) = \txtdiv f \, \text{ in }   \Omega, \quad   u_\e=0 \,\text{ on }  \pa \Omega.
\end{align*}
Then $\na u_\e \in L^p(\Omega)^d$ and
 \begin{align*}  
\| \nabla u_\e \|_{L^p(\Omega)} \le C\| f\|_{L^p(\Omega)},
\end{align*}
where $C$ depends only on $d, \Lambda, p, \Omega, N$, and $(\tau,L)$ in \eqref {smoothness}.
\end{corollary}

In the setting of quasiperiodic homogenization, the convergence rate was first obtained by Kozlov in \cite{Koz78} under a Diophantine condition (a quantitative ergodicity condition): for each row $M_i$ of $M$, $|M_i\cdot z| \ge c_0|z|^{-\beta}$ for all $z\in \Z^d\setminus \{ 0\}$ and for some $c_0,\beta>0$.
It was observed in \cite{al91} that the uniform regularity is valid under Kozlov's Diophantine condition. This condition was also required in recent work \cite{shenapde2015,AGK16,BG19} applied to the case of quasiperiodic coefficients.
In our Corollary \ref{coro1.3}, we do not need the Diophantine condition and the constant in the estimate is independent of the entries of the matrix $M$ (except for the dimension of $M$). This includes some typical degenerate cases, such as periodic coefficients with a degenerate periodic cell (e.g., a thin rectangle), quasiperiodic coefficients oscillating at two almost resonant frequencies (e.g., the ratio of the frequencies is a Liouville number), multiscale quasiperiodic coefficients, etc. 


The proof of Corollary \ref{coro1.3} based on Theorem \ref{th-w1p} is simple and we include it here. In fact, it suffices to write $A(x/\e)$ into a periodic coefficient matrix with multiple oscillating scales and apply Theorem \ref{th-w1p}.
To this end, let $M =  (M_{ij})_{1\le i\le N, 1\le j\le d}$.
Let $y_{ij} \in \R^d$ with $1\le i\le N$ and $1\le j\le d$ be $Nd$ independent variables.
Define a matrix with variables $y_{ij}$
\begin{equation*}
    A^+(y_{11},\cdots, y_{ij}, \cdots,  y_{Nd}) = B\big(\sum_{j=1}^d y_{1j}\cdot e_j, \cdots, \sum_{j=1}^d y_{Nj} \cdot e_j\big).
\end{equation*}
Then $A^+$ is 1-periodic in each $y_{ij}$ since $B(w)$ is 1-periodic. Moreover, it is straightforward to verify
\begin{equation}
    A(x/\e) = B(Mx/\e) = A^+(M_{11}x/\e, \cdots, M_{ij}x/\e, \cdots, M_{Nd}x/\e).
\end{equation}
This reduces the quasiperiodic coefficient matrix with one oscillating scale into a periodic coefficient matrix with $Nd$ oscillating scales. Clearly, in this case, all the scales are possibly not well-separated. Moreover, if $B$ is uniformly elliptic and H\"{o}lder continuous, so is $A^+$. As a result,  Corollary \ref{coro1.3} follows readily from Theorem \ref{th-w1p}.

Finally, using a similar idea of scale separation, we can show the uniform large-scale or mesoscopic-scale Lipschitz estimates. In particular, in the case of two scales, i.e., $n=2$, for any $\alpha\in (0,1)$, we have for all $\e_2^{1-\alpha} \le r\le 1$,
    \begin{equation}
        \bigg( \fint_{B_r} |\nabla u_\e|^2 \bigg)^{1/2} \le C_\alpha \bigg\{ \bigg( \fint_{B_1} |\nabla u_\e|^2 \bigg)^{1/2} + \bigg( \fint_{B_1} |F|^p \bigg)^{1/p} \bigg\},
    \end{equation}
where $C_\alpha$ depends only on $d,\Lambda,p,\tau,L$ and $\alpha$; see Theorem \ref{thm.n=2}. For general cases $n\ge 3$, we have some mesoscopic-scale Lipschitz estimates near any given scale; see Theorem \ref{thm.n>2}.

\textbf{Organization.} The rest of the paper is organized as follows. In Section \ref{sec.2}, we recall some knowledge and known results from locally periodic homogenization and reiterated homogenization. In Section \ref{sec.3}, we use the Dirichlet's theorem and reperiodization technique to separate scales and reduce one scale by a uniform approximation. The main theorem is proved in Section \ref{sec.4} by a real-variable argument. In Section \ref{sec.5}, we exploit the idea of scale separation to obtain the uniform large-scale or mesoscopic-scale Lipschitz estimates.

\textbf{Acknowledgements.} The authors would like to thank the referee for helpful comments that improve the quality of the paper. W. Niu is supported by NNSF of China (12371106, 11971031). J. Zhuge is partially supported by NNSF of China (12494541, 12288201, 12471115).

\section{Preliminaries}\label{sec.2}
\subsection{Local periodic homogenization}
\label{sec.2.1}
In this subsection, we recall some results concerning the homogenization of the locally periodic operator $-\txtdiv ( A(x,x/\e)\na)$ with a scalar $\e \in (0,1)$, which will be used in our quantitative scale reduction process in Section 3.

Suppose $A(x,y)$ satisfies the ellipticity condition \eqref{ellipticity} and is 1-periodic in $y\in \mathbb{R}^d$.
 The effective (homogenized) matrix is given by \begin{align}\label{ahat}
\widehat{A}(x)=\fint_{\mathbb{T}^d}
\Big( A(x, y) + A(x, y) \na _y \chi (x, y) \Big) dy,
\end{align}
where $\chi (x, y)= (\chi_1(x, y), \dots, \chi_d (x, y))$ is the corrector given by the cell problem
\begin{equation}\label{cell-1}
\left\{
\aligned
& -\text{\rm div}_y  \big( A  (x,   y ) \na _y  \chi_j)
=\text{\rm div}_y \big( A  (x,   y ) \na _y  y^j \big) \quad \text{ in } \mathbb{T}^d:= \R^d/\Z^d,\\
& \chi_j=\chi_j (x, y ) \text{ is 1-periodic in } y,\\
 & \int_{\mathbb{T}^d} \chi_j  (x,  y )\, dy  =0,
\endaligned
\right.
\end{equation}
for $1\le j\le d$. Here $y^j$ denotes the $j$th component of $y \in \mathbb{R}^d$.
Assume that \begin{equation}\label{lip-simple}
\|\na _x A\|_\infty
=\|\na _x A\|_{L^\infty(\mathbb{R}_x^d \times \mathbb{R}^d_y)} <\infty.
\end{equation}
The standard energy estimates for \eqref{cell-1} imply that
\begin{equation} \label{es-chi}
\|  \na_y\chi \|_{L^2(Y)} \leq C ,\quad \|\na _x  \na_y\chi \|_{L^\infty(\R_x^d; L^2(Y))} \leq C \|\na _x A\|_\infty
\end{equation} with $C$ depending only on $\Lm, d.$
 This combined with the definition of $\widehat{A}$ yields $$ \|\na _x \widehat{A}\|_\infty  \leq C \|\na _x  A \|_\infty.$$ 
 Likewise, if  
 \begin{equation}\label{Calpha.x}
     |A(x,y) - A(x', y)| \le L  |x-x'| ^{\tau}
 \end{equation}
 for some $0<\tau \le 1$ for any $x,x',y\in \R^d. $ 
 Then we have \begin{align} \label{shata}
    | \widehat{A}(x)-\widehat{A}(x')|  \leq C L |x-x'|^\tau. 
 \end{align} 

Since we do not assume any smoothness of $A(x,y)$ on $y$, a smoothing operator is needed to avoid the roughness of correctors at $\e$ scale, which has been a standard technique even in the one-scale problem (see \cite[Chapter 3.1]{shenbook2018}).
Let $\varphi \in C_{0}^{\infty}(\R^d)$ be supported in $B_{1/2}(0)$ such that $\varphi\geq 0$ and $\int_{\mathbb{R}^{d}}\varphi  dx=1$.
  For functions of the form $g^\e (x)= g(x, x/\e)$, we
define  a partial-smoothing operator 
\begin{align}\label{smoothing}
  S_\e (g^\e)(x)
  =\e^{- d}\int_{\mathbb{R}^d}
  g(z, x/\e)\varphi  ((x-z)/\e) dz.
\end{align}


The following theorem  provides  the convergence rate in $H^1(\Omega)$ for the locally periodic operator. The proof may be found in \cite[Theorem 3.1 and Remark 3.1]{nsxjfa2020}.
\begin{theorem} \label{th-app-1}
Let $\Omega$ be a bounded Lipschitz domain in $\mathbb{R}^d$. Let $u_\e$ be the solution to $ -\txtdiv (A(x,x/\e) \na  u_\e) =F$ in $\Omega
$ with $u_\e =g $ on $ \partial\Omega,$ and $u_0$ the solution to  $ -\txtdiv (\widehat{A}(x)  \na  u_0) =F$ in $\Omega
$ with $u_0 =g $ on $ \partial\Omega.$ Then we have
\begin{align} \label{th-app-1-re}
\begin{split}
 &\| u_\e - u_0 \|_{L^2(\Omega)} + \|  \na  u_\e-\na  u_0- S_\e((\na _y \chi )^\e \na  u_0)\|_{L^2(\Omega)}\\
   &\leq C \e \big\{ (1+
 \|\na _x A\|_\infty)   \|\na  u_0\|_{L^2(\Omega)}
+ \|\na ^2 u_0\|_{L^2(\Omega\setminus \Omega_{3\e} )} \big\}
+ C \| \na  u_0\|_{L^2(\Omega_{4\e})},
\end{split}
\end{align}
where we have extended $u_0$ to the whole space $\R^d,$  $\Omega_t =\big\{ x\in \Omega:  \text{\rm dist}(x, \partial\Omega)< t \big\}$, and $C$ depends only on $d$, $\Lm$ and $\Omega$.
\end{theorem}

We emphasize that in the above theorem, $A$ has no smoothness assumption in $y$.


\subsection{Reiterated correctors and effective matrix}

Consider the operator  with $n$ scales $ -\text{\rm div} ( A(x/\e_1, x/\e_2,\cdots, x/\e_n) \na )$, where $A$ satisfies the ellipticity and periodicity conditions  \eqref{ellipticity} and \eqref{periodicity}. Assume  that  $\e_n $ is separated from $\e_i$ for $1\le i\le n-1.$ By using the idea of reiterated homogenization, one  can view $(\e_i)_{1\le i\le n-1}$ as parameters and homogenize the finest scale $\e_n$ to obtain an operator of the same type but with only $n-1$ scales.

Indeed rewrite 
$$\mathcal{A}_\e (x,y):=A(x/\e_1, x/\e_2,\cdots, x/\e_{n-1},y).$$ 
It is obvious that $ \mathcal{A}_\e  (x,y) $  satisfies \eqref{ellipticity}, and is locally 1-periodic in $y$. 
Let $ \chi_\e  (x,y) $ be the corrector given by \eqref{cell-1} with $A(x,y)$ replaced by $ \mathcal{A}_\e(x,y)$, and  $\widehat{\mathcal{A}}_\e(x)$ the effective matrix of $ \mathcal{A}_\e $ given by \eqref{ahat}.
In view of the structure of $\mathcal{A}_\e(x,y)$, the corresponding corrector $\chi_\e(x,y)$ takes a form of
\begin{equation} \label{mathcalX}
    \chi_\e (x,y)=\mathcal{X} (x/\e_1, \cdots,x/\e_{n-1},  y ),
\end{equation}
where the reiterated corrector $\mathcal{X} = (\mathcal{X}_j)$ is defined by the cell problem
\begin{equation}\label{cell-n}
\left\{
\aligned
& -\text{\rm div}_y  \big( A  (y_1,\cdots, y_{n-1},   y ) \na _y  \mathcal{X}_j)
=\text{\rm div}_y \big( A  (y_1,\cdots, y_{n-1},   y ) \na _y  y^j \big) \quad \text{ in } \mathbb{T}^d,\\
& \mathcal{X}_j=\mathcal{X}_j (y_1,\cdots, y_{n-1}, y ) \text{ is 1-periodic in } y,\\
 & \int_{\mathbb{T}^d} \mathcal{X}_j  (y_1,\cdots, y_{n-1},  y )\, dy  =0,
\endaligned
\right.
\end{equation}
for $1\le j\le d$,  which is also 1-periodic in $y_1,\cdots, y_{n-1}$, due to the periodicity of $A$. Moreover,  $\widehat{\mathcal{A}}_\e(x)$ 
  takes a form of
\begin{equation}\label{formA}
     \widehat{\mathcal{A}}_\e(x)  = \widehat{\mathcal{A}}(x/\e_1, \cdots, x/\e_{n-1})
\end{equation}
with
\begin{equation}\label{def.hatA}
    \widehat{\mathcal{A}}(y_1,y_2,\cdots, y_{n-1}) = \fint_{\mathbb{T}^d}
   A(y_1, \cdots,y_{n-1}, y) \Big( I + \na _y \mathcal{X} (y_1, \cdots,y_{n-1}, y) \Big) dy.
\end{equation}
 It is obvious that $\widehat{\mathcal{A}}(y_1,y_2,\cdots, y_{n-1})$ is 1-periodic in each $y_i$ with $1\le i\le n-1$, and it is defined independent of $(\e_1,\cdots,\e_n)$. 
 Similar to the case $n=1$,  $\widehat{\mathcal{A}}$ satisfies the ellipticity condition \eqref{ellipticity}. Furthermore, similar to \eqref{shata}, if $A$ is H\"{o}lder continuous in $y_i$ for $1\le i\le n-1$,  so is $\widehat{\mathcal{A}}$. 

\section{Quantitative scale separation and scale reduction}
\label{sec.3}

\subsection{Scale separation}\label{sec3.1}

The key ingredient of this paper is the Dirichlet's theorem on the simultaneous Diophantine approximation. The connection between the Dirichlet's theorem and the regularity theory in homogenization is previously unknown. The only loosely related notion is the Diophantine condition imposed in the setting of quasiperiodic homogenization, as mentioned earlier. The Dirichlet's theorem appears to be a powerful tool (instead of a condition) that eventually allows us to derive the uniform Calder\'{o}n-Zygmund estimates in multiscale or quasiperiodic homogenization without any additional conditions.

Let us first recall the Dirichlet's theorem; see \cite{Sch80}.
\begin{theorem}[Dirichlet (1842)]\label{th.Dirichlet}
    Suppose that $\alpha_1,\alpha_2,\cdots,\alpha_m$ are $m$ real numbers and $Q>1$. There exist integers $q, p_1, p_2,\cdots, p_m$ such that $1\le q< Q^m$ and
    \begin{equation}
        \sup_{1\le i\le m}\Big|\alpha_i - \frac{p_i}{q}\Big| < \frac{1}{qQ }.
    \end{equation}
\end{theorem}
Roughly speaking, the Dirichlet's theorem states that any real numbers can be approximated by good rational numbers with a quantitative small error. 
In the rest of this subsection, we apply the Dirichlet's theorem to separate at least one scale in the coefficient matrix oscillating at multiple (possibly unseparated) scales.

Without loss of generality,  hereafter we always assume 
\begin{equation}\label{cond.eps.order}
    1\ge \e_1 \ge \e_2 \ge \cdots \ge \e_n > 0.
\end{equation}



Let $Q \gg 1$ be a large number. We emphasize that the fixed number $Q$ will be chosen later independent of $\{ \e_1, \e_2,\cdots, \e_n \}$. We say $a $ is $Q$-separated from  $b$ if $b  \ge Q a$. This scale separation condition is crucial in quantitative reiterated homogenization.

Given $Q>1$ as above, by the Dirichlet's theorem with $\alpha_i = \e_n/\e_i \in (0,1]$ for $i=1,2,\cdots, n-1$, we can find integers $q, p_1, p_2,\cdots, p_{n-1}$ such that $1\le q < Q^{n-1}$ and
\begin{equation}\label{est.Diop.Approx}
    \sup_{1\le i\le n-1} \Big|\frac{\e_n}{\e_i} - \frac{p_i}{q}
    \Big| < \frac{1}{q Q}.
\end{equation}
Moreover, because of \eqref{cond.eps.order}, we have $0\le p_i\le p_{i+1}  \le q < Q^{n-1}$. Note that it is possible that $p_i = 0$ if $\e_i$ is large enough that $\e_n/\e_i < 1/Q$ (i.e., $\e_n$ is already $Q$-separated from $\e_i$).

Define, for $1\le i\le n-1$,
\begin{equation}\label{gami}
    \gamma_i := \Big| \frac{\e_n}{\e_i} - \frac{p_i}{q} \Big| \quad \text{and} \quad s_i = \sgn\Big(\frac{\e_n}{\e_i} - \frac{p_i}{q}\Big).
\end{equation}
Then, we can write
\begin{equation}\label{eq.rearrange}
    \frac{1}{\e_i} = \frac{\gamma_i s_i}{\e_n} + \frac{p_i}{q} \frac{1}{\e_n}.
\end{equation}
Now, we define
\begin{equation}\label{A-sep}
    A^\sharp (y_1, y_2,\cdots, y_n) = A(s_1 y_1 + p_1 y_n, s_2 y_2 + p_2 y_n, \cdots, s_{n-1} y_{n-1} + p_{n-1} y_n, q y_n).
\end{equation}
The key insight here is that,
since $s_i = \pm 1$ and $p_1, p_2,\cdots, p_n$ are all integers, $A^\sharp$ is 1-periodic in each $y_i$. We thereby will call the transformation from $A$ to $A^\sharp$ a technique of reperiodization. As a result of \eqref{eq.rearrange} and \eqref{A-sep},  we have
\begin{equation}\label{eq.rewrite}
    A\big(\frac{x}{\e_1}, \cdots, \frac{x}{\e_n}\big) = A^\sharp\big(\frac{\gamma_1 x}{\e_n}, \frac{\gamma_2 x}{\e_n}, \cdots, \frac{\gamma_{n-1} x}{\e_n}, \frac{x}{q\e_n}\big).
\end{equation}
In other words, the  reperiodization allows us to rewrite the original periodic matrix $A$ oscillating at scales $\{ \e_1,\e_2,\cdots, \e_n \}$ into a new periodic matrix $A^\sharp$ oscillating at different scales $\{\e_n/\gamma_1, \e_n/\gamma_2, \cdots, \e_n/\gamma_{n-1}, q\e_n  \}$. Note that if $p_i = 0$ for some $i$, then the corresponding scale for that (slower) variable does not change, i.e., $\e_n/\gamma_i = \e_i$. If $\gamma_i = 0$ for some $i$, then we just do not have the scale $\e_n/\gamma_i$ and we have less oscillating scales after reperiodization (which is even better). Without loss of generality, we may assume this does not happen.

The new scales and the new 1-periodic matrix $A^\sharp$ have the following crucial properties:
\begin{itemize}
    \item The smallest scale $q \e_n$ is $Q$-separated from  $\e_n/\gamma_i$ for each $1\le i\le n-1$.
In fact, \eqref{est.Diop.Approx} implies
\begin{equation}\label{Q-separation}
    \frac{\e_n/\gamma_i}{q\e_n} = \frac{1}{q\gamma_i} \ge Q.
\end{equation} 
This is the key property for our application.

\item If $A$ is H\"{o}lder continuous in $y_i$ for $1\leq i\leq n-1$, i.e.,
 \begin{equation}\label{smoothness-n}
 \begin{aligned}
     &|A(y_1,\cdots,y_{n-1}, y_n) - A(y_1',\cdots,y_{n-1}',y_n)| \\
     & \le L \big\{|y_1'-y_1|+\cdots +|y_{n-1}'-y_{n-1}|\big\}^\tau
 \end{aligned}
	\end{equation}
 with $0<L, 0<\tau\leq 1$ for any $(y_1,\cdots,y_n), (y_1',\cdots,y_{n-1}',y_n)\in \R^{d\times n}$, 
 then $A^\sharp$ is also H\"{o}lder continuous in $y_i$ for $1\le i\le n-1$ with the same constants $(\tau, L)$. However, $A^\sharp$ may lose good regularity in $y_n$ since many periods have been compressed into a single 1-periodic cell.

\item  The construction of $A^\sharp$ relies on the choice of $Q$ as well as the ratios $\e_n/\e_i$ with $1\le i\le n-1$. In other words, $A^\sharp$ is scale-invariant. Also note that $Q$ is a dimensionless parameter.
\end{itemize}


It is important to note that the above scale separation process can only separate the smallest scale from the rest $n-1$ larger scales, while the relationships among the rest $n-1$ larger scales cannot be determined. This allows us to perform a one-scale reduction by using the idea of  reiterated homogenization. 

To see more clearly how the above argument works, we provide a concrete example of two scales. Let $A_\e(x) = A(\frac{x}{\e}, \frac{x}{\e(1/3 + \delta)})$, where $\e \ll 1, \delta\ll 1$ (e.g., $\delta = \e^{1/2}$). Clearly, in this case the two scales $\e_1 = \e$ and $\e_2 = \e(1/3+\delta)$ are not separated since their ratio is $\alpha = \e_2/\e_1 = 1/3 + \delta$. Now, note that this ratio can be well-approximated by the rational number $1/3$ in the sense of Dirichlet's theorem for $p =1, q = 3$ and up to $Q = (3\delta)^{-1} \gg 1$. Using the previous argument, we can write
\begin{equation}
    A\Big(\frac{x}{\e}, \frac{x}{\e(1/3 + \delta)} \Big) = A^\sharp\Big(\frac{\delta x}{\e(1/3 + \delta)}, \frac{x}{\e + 3\e \delta} \Big),
\end{equation}
where $A^\sharp(y_1, y_2) = A(y_1 + y_2, 3y_2)$ is 1-periodic in both $y_1$ and $y_2$. Moreover, the new scales $\e \delta^{-1}(1/3 + \delta)$ and $\e + 3\e \delta$ are obviously well-separated as their ratio is exactly $Q = (3\delta)^{-1}$.

\subsection{Scale reduction and uniform approximation}\label{sec3.2}
In this subsection we use the scale separation technique in Section \ref{sec3.1} to reduce at least one scale.  Precisely, we shall prove that the solution of \eqref{operaotr-n} can be uniformly approximated by a solution to the equation of the same type but with at most $n-1$ oscillating scales. The following is the main theorem.


\begin{theorem}\label{th-app}
     Suppose that $A$ satisfies the assumptions \eqref{ellipticity}, \eqref{periodicity} and \eqref{smoothness-n}. Let $u_\e$ be a weak solution to
     \begin{align}\label{th-app-eq1}
     -\txtdiv \big( A(x/\e_1,x/\e_2,\cdots, x/\e_n) \na u_\e \big)=F  \quad \text{in } B_{2r}:= B(0,2r).
     \end{align}
     Then for any $Q>1$, there exists a coefficient matrix $ A^\flat = A^\flat(y_1,\cdots, y_{n-1})$, scales $\e_k'\in (0,\infty) $ with $ k=1,\cdots,n-1$, and a weak solution $u^\flat_{\e'}$  to
     \begin{equation}\label{eq.n-1scale}
         -\txtdiv \big( A^\flat(x/\e'_1,x/\e'_2,\cdots, x/\e'_{n-1}) \na u^\flat_{\e'} \big)=F  \quad \text{in } B_r,
     \end{equation}
       such that 
       \begin{equation}
           \|\na u^\flat_{\e'} \|_{L^2(B_r)}  \leq  C \big\{ \|\na u_\e \|_{L^2(B_{2r})} + r\|F\|_{L^2(B_{2r})} \big\},
       \end{equation}
       and
    \begin{equation}\label{th-app-re}
        \| \na u_\e - \na u^\flat_{\e'} -U_{\e'}\|_{L^2(B_r)} \le C  \bigg\{ \bigg(\frac{Q^{n-1}  \e_n}{ r}\bigg) ^\sigma+ Q^{-\tau} \bigg\}  \big\{\| \na u_\e \|_{L^2(B_{2r})} +r\|F\|_{L^2(B_{2r})} \big\}
    \end{equation} for some $\sigma>0,$
    where $ U_{\e'} $ is given by
    \begin{equation} \label{Ue}
        U_{\e'}(x) = (\e_n')^{- d}\int_{\mathbb{R}^d}
        \varphi  \bigg(\frac{x-z}{ \e_n'}\bigg)  (\na_y\mathcal{X}^\sharp) \bigg(\frac{z}{\e_1'},\cdots,\frac{z}{\e'_{n-1}}, \frac{x}{\e_n'}\bigg)  \na u^\flat_{\e'}(z)  dz,
    \end{equation}
    $\mathcal{X}^\sharp$ is the corrector given by \eqref{cell-n} with $A$ replaced by $A^\sharp$, and $\e'_n$ is given by \eqref{ei'}.
    Moreover, the constant $C$ depends only on $d, n, \Lm$ and  $(\tau,L)$ in \eqref{smoothness-n}, and the assumptions \eqref{ellipticity}, \eqref{periodicity} and \eqref{smoothness-n} are preserved for $A^\flat$. 
\end{theorem}

\begin{remark}\label{rmk.rate}
    Note that in the above theorem, we only state the approximation error for the gradient $\nabla u_\e$. Actually, we also have the same error bound for $\| u_\e - u_{\e'}^\flat \|_{L^2(B_r)}$. In particular, if we pick $Q = (r/\e_n)^{\theta/(n-1)}$ for some $\theta\in (0,1)$, then we have for some $\gamma \in (0,1)$,
    \begin{equation}
        \| u_\e - u_{\e'}^\flat \|_{L^2(B_r)} \le Cr \Big( \frac{\e_n}{r} \Big)^\gamma  \big\{\| \na u_\e \|_{L^2(B_{2r})} +r\|F\|_{L^2(B_{2r})} \big\}.   
    \end{equation}
    This error estimate shows that, at any scale $r\gg \e_n$, the solution $u_\e$ of the problem \eqref{th-app-eq1} with $n$ oscillating scales can be well approximated by the solution of \eqref{eq.n-1scale} with (at most) $n-1$ oscillating scales. Moreover, in general the coefficient matrix $A^\flat$ (depending on $Q$) depends on $A$ and $r/\e_n$. In other words, the approximate equation \eqref{eq.n-1scale} may vary according to different radii $r$, which is quite different from the classical homogenization theory in which the homogenized equation is fixed (independent of $r$). This is essentially due to the absence of scale-separation condition between $\e_n$ and the rest of the scales. 
\end{remark}

\begin{remark}
    One may naturally ask if we can repeat Theorem \ref{th-app} to get the next approximation with at most $n-2$ oscillating scales for the equation \eqref{eq.n-1scale}. This is unfortunately impossible in general, because in the worst case all the scales $\e_1', \e_2',\cdots, \e_{n-1}'$ will not be small and we are not able to get a small error in the next step.
\end{remark}

The proof of Theorem \ref{th-app} relies on the scale separation technique in Section \ref{sec3.1} and the following approximation theorem that pertains to the quantitative reiterated homogenization.
\begin{theorem}\label{th-app-n}
     Suppose $A$ satisfies the assumption \eqref{ellipticity}, \eqref{periodicity}, and \eqref{smoothness-n}.
       Assume  $\e_{n}$ is $Q$-separated from $\e_k, k=1,\cdots,n-1$. Let $u_\e$ be a weak solution to  
       \begin{equation}\label{eq.Qsepa}
       -\txtdiv \big( A(x/\e_1,x/\e_2,\cdots, x/\e_n) \na u_\e \big)=F  \quad\text{ in } B_2.
       \end{equation}
       Then there exist a 1-periodic matrix $\widehat{\mathcal{A}} = \widehat{\mathcal{A}}(y_1,\cdots, y_{n-1})$ with $n-1$ scales and a weak solution $\widehat{u}_\e$ to 
       $$-\txtdiv \big( \widehat{\mathcal{A}}(x/\e_1,x/\e_2,\cdots, x/\e_{n-1}) \na \widehat{u}_\e \big)=F  \quad\text{ in } B_1,$$ 
    such that 
       \begin{equation}\label{hatu.energy}
           \|\na \widehat{u}_\e \|_{L^2(B_1)}  \leq  C \big\{ \|\na u_\e \|_{L^2(B_2)} + \|F\|_{L^2(B_{2})} \big\},
       \end{equation}
       and
    \begin{equation}\label{th-app-n-re1}
        \| \na u_\e - \na \widehat{u}_\e -S_{\e_n}( (\na_y \mathcal{X})^{\e} \na \widehat{u}_\e)\|_{L^2(B_1)} \le C  ( Q^{-\tau} + \e_n^\sigma) \big\{\| \na u_\e \|_{L^2(B_2)} + \|F \|_{L^2(B_2)}  \big\},
    \end{equation} 
    for some $\sigma>0$, where $\mathcal{X}$ and $\widehat{\mathcal{A}}$ are defined by \eqref{cell-n} and \eqref{def.hatA}, respectively,
    and 
    \begin{equation}
        S_{\e_n}( (\na_y \mathcal{X})^{\e} \na \widehat{u}_\e)(x) = (\e_n)^{- d}\int_{\mathbb{R}^d}
        \varphi  \bigg(\frac{x-z}{ \e_n}\bigg)  (\na_y \mathcal{X}) \bigg(\frac{z}{\e_1},\cdots,\frac{z}{\e_{n-1}}, \frac{x}{\e_n}\bigg)  \na \widehat{u}_{\e}(z)  dz.
    \end{equation}   
    Moreover, the constant $C$ depends only on $d, n, \Lm$ and $(\tau,L)$ in  \eqref{smoothness-n}.
\end{theorem}
We first prove Theorem \ref{th-app}, assuming Theorem  \ref{th-app-n}.
\begin{proof}[Proof of Theorem \ref{th-app}]
We first consider the case $r=1$. Given $ 1\ge \e_1 \ge \e_2 \ge \cdots \ge \e_n > 0$ and $Q>1$, by the Dirichlet's theorem, we can find integers $q, p_1, p_2,\cdots, p_{n-1}$ such that $1\le q < Q^{n-1}$, and \eqref{est.Diop.Approx}-\eqref{Q-separation} hold. Set 
\begin{equation}\label{ei'}
    \begin{aligned}
        \e_i' & = \e_i/\gamma_i, \quad \text{for } i=1,2,\cdots, n-1, \\
        \e_n' &= q\e_n.
    \end{aligned}
\end{equation}
By  \eqref{eq.rewrite} and \eqref{th-app-eq1}, we know that $u_\e$ satisfies
 \begin{align}\label{p-thw1p-3}
     -\txtdiv \big( A^\sharp(x/\e'_1, \cdots,   x/\e'_n) \na u_\e \big)= F \,\text{ in }  B_{2}.
     \end{align}
By \eqref{Q-separation}, $\e_n'$ is $Q$-separated from $\e_i'$ for all $i=1,2,\cdots, n-1$. Thus, by Theorem \ref{th-app-n}, there exist  $\widehat{\mathcal{A}}^\sharp = \widehat{\mathcal{A}}^\sharp(y_1,\cdots, y_{n-1})$ given by \eqref{def.hatA} with $A$ replaced by $A^\sharp$, and a weak solution $\widehat{u}_{\e'}$ to 
\begin{equation}
    -\txtdiv \big( \widehat{\mathcal{A}}^\sharp (x/\e'_1, \cdots,   x/\e'_{n-1}) \na \widehat{u}_{\e'} \big)= F  \quad \text{in } B_{1},
\end{equation}
  such that  
  \begin{equation}
      \| \na \widehat{u}_{\e'} \|_{L^2(B_{1})}\leq C \{ \| \na u_\e \|_{L^2(B_2)} +\| F \|_{L^2(B_2)}\},
  \end{equation}
  and
  \begin{equation}
  \begin{aligned}
        &\| \na u_\e - \na \widehat{u}_{\e'} -S_{\e_n'}((\na_y\mathcal{X}^\sharp )^{\e'} \na \widehat{u}_{\e'} )\|_{L^2(B_1)} \\
        &\le C  (  Q^{-\tau}  + (\e'_n)^\sigma)   \{ \| \na u_\e \|_{L^2(B_2)} +\| F \|_{L^2(B_2)}\}  \\
        &\le C  (  Q^{-\tau}  + (Q^{n-1}\e_n)^\sigma)   \{ \| \na u_\e \|_{L^2(B_2)} +\| F \|_{L^2(B_2)}\},
    \end{aligned}
    \end{equation} 
    for some $\sigma>0$,  where $\mathcal{X}^\sharp$ is the corrector defined in \eqref{cell-n} with $A$ replaced by $A^\sharp$, and
    \begin{align*}
    \begin{split}
     S_{\e_n'}((\na_y\mathcal{X}^\sharp)^{\e'} \na \widehat{u}_{\e'} )
    &=(\e_n')^{- d}\int_{\mathbb{R}^d}
   \varphi  \bigg(\frac{x-z}{ \e_n'}\bigg)  (\na_y \mathcal{X}^\sharp) \left(\frac{z}{\e_1'},\cdots,\frac{z}{\e'_{n-1}}, \frac{x}{\e_n'}\right)  \na \widehat{u}_{\e'}(z)  dz.
   \end{split}
    \end{align*} 
    We rename $A^\flat = \widehat{\mathcal{A}}^\sharp, u^\flat_{\e'} = \widehat{u}_{\e'}$, $U_{\e'} = S_{\e_n'}((\na_y \mathcal{X}^\sharp)^{\e'} \na \widehat{u}_{\e'} )$, and derive the desired estimate in the case $r=1$.
    {Finally, note that $A^\sharp$ and $A^\flat$ are scale-invariant. The estimate for general $r>0$ follows  immediately by rescaling. The proof is complete.}

    \end{proof}

We now provide the proof of Theorem \ref{th-app-n}, following the idea of \cite{nsxjfa2020}.

\begin{proof}[Proof of Theorem \ref{th-app-n}]
We consider the matrix 
\begin{equation}\label{calA}
\mathcal{A}_\e (x,y):=A(x/\e_1, x/\e_2,\cdots, x/\e_{n-1},y).
\end{equation} 
{Recall that  $\mathcal{A}_\e (x,y)$ is strongly elliptic and locally 1-periodic in $y$. }
In view of \eqref{smoothness},  $\mathcal{A}_\e(x,y)$ is  H\"{o}lder continuous in $x$, and 
\begin{align}\label{ma-holder}
 | \mathcal{A}_\e  (x',y)- \mathcal{A}_\e  (x,y)| \leq L\sum_{k=1}^{n-1}\e_{k}^{-\tau} |x-x'|^\tau.
\end{align}
In order to apply Theorem \ref{th-app-1}, we need to find an approximate matrix $ \mathcal{A}^{\rm app}_\e  = \mathcal{A}^{\rm app}_\e (x, y)$ which is Lipschitz in the $x$ variable. In fact, we define
\begin{equation}
    \mathcal{A}^{\rm app}_\e (x, y) = \e_n^{-d} \int_{\R^d} \varphi\Big( \frac{x-z}{\e_n} \Big) \mathcal{A}_\e(z,y) dz,
\end{equation}
with $\varphi$ given as in Section \ref{sec.2.1}.
Then  
$ \mathcal{A}^{\rm app}_\e $ satisfies the ellipticity condition \eqref{ellipticity}, is 1-periodic in $y$, and
\begin{equation}\label{masmothing-holder}
\|  \mathcal{A}_\e  - \mathcal{A}_\e^{\rm app}  \|_\infty
\le C L \sum_{k=1}^{n-1}\e_{k}^{-\tau} \e_{n} ^\tau
\quad
\text{ and } \quad
\|\nabla_x  \mathcal{A}_\e^{\rm app}  \|_\infty
\le C L \sum_{k=1}^{n-1}\e_{k}^{-\tau} \e_{n} ^{\tau-1},
\end{equation}
where  $C$ depends only on $d$ and $\tau$. 

Let $u_\e$ be the weak solution to \eqref{eq.Qsepa}. Then by \eqref{calA}, it satisfies
\begin{equation}\label{ueinb2}
-\txtdiv \big( \mathcal{A}_\e  (x,  x/\e_n)\nabla u_\e \big)  =F \quad
\text{ in } B_{2}.
\end{equation}
Let $u_\e^{\rm app}$ be the weak solution to
\begin{equation}\label{eqve}
-\text{\rm div} \big(  \mathcal{A}^{\rm app}_\e  (x, x/\e_n)\nabla u_\e^{\rm app}\big) =F \quad
\text{in } B_{3/2}, \quad \text{and} \quad
u_\e^{\rm app} =u_\e \quad \text{on } \partial B_{3/2}.
\end{equation}
Combining \eqref{ueinb2}, \eqref{eqve} and the first inequality in \eqref{masmothing-holder}, we can apply the energy estimate to the equation of $u_\e - u_\e^{\rm app}$ 
\begin{equation}
\left\{
\begin{aligned}
    & -\text{\rm div} \big(  \mathcal{A}^{\rm app}_\e  (x, x/\e_n)\nabla (u_\e^{\rm app} - u_\e) \big) = \txtdiv \big( (\mathcal{A}^{\rm app}_\e(x,x/\e_n) - \mathcal{A}_\e(x,x/\e_n)) \nabla u_\e \big)  \quad
\text{in } B_{3/2}, \\
& u_\e^{\rm app} - u_\e = 0 \quad \text{on } \partial B_{3/2},
\end{aligned} \right.
\end{equation}
to obtain
\begin{align}\label{ueve-error}
\int_{B_{3/2}}
|\nabla (u_\e -u_\e^{\rm app})|^2
&\le C L^2 \sum_{k=1}^{n-1} \e_{k}^ {-2\tau} \e_n^{2\tau}
\int_{B_{3/2}} |\nabla u_\e|^2.
\end{align}

Next, we apply Theorem \ref{th-app-1}  to the equation \eqref{eqve}.
Let $ \chi^{\rm app}_\e(x,y)$ be the corrector given by  \eqref{cell-1} with $A(x,y)$ replaced by $ \mathcal{A}^{\rm app}_\e  (x, y)$, and $ \widehat{\mathcal{A}}^{\rm app}_\e (x)$ the corresponding effective coefficient matrix given by \eqref{ahat}.
Let $\widehat{u}_\e^{\rm app}$ be the solution to 
\begin{align}\label{eqv0}
-\text{\rm div} \big(  \widehat{\mathcal{A}}^{\rm app}_\e (x) \nabla \widehat{u}_\e^{\rm app} \big)=F \quad \text{in } B_{5/4},  \quad \text{and}\quad \widehat{u}_\e^{\rm app}= u_\e^{\rm app} \quad \text{on } \pa  B_{5/4}  .\end{align}
Thanks to Theorem \ref{th-app-1},
\begin{equation}\label{vev0-error}
\aligned
& \int_{B_{5/4}}
|\na u_\e^{\rm app} - \na \widehat{u}_\e^{\rm app}- S_{\e_n}( (\na_y  \chi^{\rm app}_\e)^{\e_n} \na u_\e^{\rm app})|^2  \\
& \leq C \e_n^2  (1+
 \|\na_x\mathcal{A}_\e^{\rm app} \|_\infty)^2  \int_{B_{5/4}}
|\na \widehat{u}_\e^{\rm app}|^2
\\
& \qquad + C\e_n^2\int_{B_{(5/4)-3\e_n}}
 |\na ^2 \widehat{u}_\e^{\rm app} |^2
+ C\int_{B_{5/4}\setminus B_{(5/4)-4\e_n}}
 |\na  \widehat{u}_\e^{\rm app} |^2.
\endaligned
\end{equation}
By the H\"older's inequality, the last term on the right-hand side of  \eqref{vev0-error}
is bounded by
\begin{align}\label{2.26}
\int_{B_{5/4}\setminus B_{(5/4)-4\e_n}}
 |\na  \widehat{u}_\e^{\rm app} |^2\leq  C \e_n ^{1-\frac{2}{q} } \bigg(\int_{B_{5/4}} |\nabla \widehat{u}_\e^{\rm app}|^q \bigg)^{2/q} \quad  \text{ for } q>2.
\end{align}
On the other hand, by the interior $H^2$ estimate for the equation \eqref{eqve}, for any $B(z,2\rho) \subset B_{3/2}$,
\begin{equation*}
\int_{B(z,\rho)}
|\nabla^2 \widehat{u}_\e^{\rm app}|^2
\le C \int_{B(z, 2\rho)}
|F|^2
+ C \big( \|\na_x\mathcal{A}^{\rm app}_\e \|_\infty^2 +\rho^{-2}
\big)
\int_{B(z, 2\rho)} |\nabla \widehat{u}_\e^{\rm app}|^2.
\end{equation*}
Then it follows by a covering argument that
\begin{equation}\label{2.27}
\aligned
& \int_{B_{(5/4) -3\e_n}}
|\nabla^2 \widehat{u}_\e^{\rm app}|^2\, dx
\\
&\leq  C \int_{B_{5/4}}
|F|^2
+ C \|\na_x\mathcal{A}^{\rm app}_\e  \|^2_\infty \int_{B_{5/4}} |\nabla \widehat{u}_\e^{\rm app} |^2 + \e_n ^{-\frac{2}{q}-1} \bigg(\int_{B_{5/4}} |\nabla \widehat{u}_\e^{\rm app}|^q \bigg)^{2/q},
\endaligned
\end{equation}
where the H\"{o}lder's inequality has been used for the last integral.
By Meyers' estimate for the elliptic equations \eqref{ueinb2}, \eqref{eqve} and \eqref{eqv0}, we have for some $q>2$, 
 \begin{align}\label{2.28}
 \begin{split} &\bigg(\int_{B_{3/2}} |\nabla u_\e|^q \bigg)^{2/q} \leq  C \int_{B_2}
|F|^2 + C \int_{B_{2}} |\nabla u_\e|^2,  \\
 &\bigg(\int_{B_{3/2}} |\nabla u_\e^{\rm app}|^q \bigg)^{2/q} \leq  C \int_{B_{3/2}}
|F|^2 + C\bigg(\int_{B_{3/2}} |\nabla u_\e|^q \bigg)^{2/q}, \\
 &\bigg(\int_{B_{5/4}} |\nabla \widehat{u}_\e^{\rm app}|^q \bigg)^{2/q} \leq  C \int_{B_{5/4}}
|F|^2 + C \bigg(\int_{B_{5/4}} |\nabla u_\e^{\rm app}|^q \bigg)^{2/q},\end{split}\end{align}
   which implies that
   \begin{align} \label{2.29}
 &\bigg(\int_{B_{5/4}} |\nabla \widehat{u}_\e^{\rm app}|^q \bigg)^{2/q} \leq  C \int_{B_2}
|F|^2 + C  \int_{B_{2}} |\nabla u_\e|^2 . \end{align}
By taking \eqref{2.26} and \eqref{2.27} into  \eqref{vev0-error}, and using \eqref{2.29} we obtain that
\begin{equation}\label{vev0-error-1}
\aligned
&  \int_{B_{5/4}}
|\na u_\e^{\rm app} - \na \widehat{u}_\e^{\rm app}- S_{\e_n}( (\na_y  \chi_\e^{\rm app})^{\e_n} \na \widehat{u}_\e^{\rm app})|^2  \\
& \le C \big\{ \e_n^{1-\frac{2}{q}}
+  L^2\sum_{k=1}^{n-1}\e_{k}^{-2\tau} \e_n^{2\tau} \big\}
\bigg\{  \int_{B_{2}}
|\na u_\e|^2
+  \int_{B_{2}} |F|^2
\bigg\}.
\endaligned
\end{equation}

{Now, let $ \chi_\e  (x,y) $ be the corrector given by \eqref{cell-1} with $A(x,y)$ replaced by $ \mathcal{A}_\e(x,y)$, and  $\widehat{\mathcal{A}}_\e(x)$ the effective matrix of $ \mathcal{A}_\e $ given by \eqref{ahat} (recall that $ \chi_\e  (x,y) $ and $\widehat{\mathcal{A}}_\e(x)$  take the forms of \eqref{mathcalX} and \eqref{formA}).} 
Note that $\chi_\e - \chi_\e^{\rm app} = (\chi_{\e,j} - \chi_{\e,j}^{\rm app})$ satisfies
\begin{align*}
&-\text{\rm div}_y  \big(  \mathcal{A}_\e(x,   y ) ( \na _y  \chi_{\e,j}-\na_y  \chi_{\e,j}^{\rm app} )\big)\\
&=\text{\rm div}_y \big( ( \mathcal{A}_\e (x,   y )- \mathcal{A}_\e^{\rm app}  (x, y ))\na _y  \chi^{app}_{\e,j}\big)
 +\text{\rm div}_y \big(  (\mathcal{A}_\e(x,   y )- \mathcal{A}_\e^{\rm app}  (x, y )) e_j \big) \quad \text{ in } \mathbb{T}^d.
\end{align*}
Thus, the standard energy estimate and the first inequality in  \eqref{masmothing-holder} imply that 
\begin{align}\label{chi.error}
\sup_{x\in \R^d}\|\na _y  \chi_{\e,j}(x,\cdot)-\na_y  \chi^{\rm app}_{\e,j}(x,\cdot) \|_{L^2(\mathbb{T}^d)}
\le C L \sum_{k=1}^{n-1}\e_{k}^{-\tau} \e_{n} ^\tau,
\end{align}
where $C$ depends only on $d,\Lm$ and $\tau$. 
This together with the definitions of $\widehat{\mathcal{A}}_\e^{\rm app}$ and $\widehat{ \mathcal{A}}_\e$  gives 
\begin{equation}\label{Aapp-A}
\|  \widehat{\mathcal{A}}_\e^{\rm app}   - \widehat{ \mathcal{A}}_\e \|_\infty
\le C L \sum_{k=1}^{n-1}\e_{k}^{-\tau} \e_{n} ^\tau .
\end{equation}
Let $\widehat{u}_\e$ be the weak solution to
\begin{align}\label{equ0}
-\text{\rm div} \big(\widehat{ \mathcal{A}}_\e (x)\nabla \widehat{u}_\e \big)=F \quad \text{ in }  B_{5/4}, \quad  \text{ and } \quad \widehat{u}_\e= \widehat{u}^{\rm app}_\e  \quad \text{ on } \partial B_{5/4}.
 \end{align} 
 Similar to \eqref{ueve-error}, we apply
the standard energy estimate to the equation for $\widehat{u}_\e^{\rm app} - \widehat{u}_\e$,
\begin{equation}
\left\{
\begin{aligned}
    & -\text{\rm div} \big(  \widehat{\mathcal{A}}_\e  (x, x/\e_n)\nabla (\widehat{u}_\e - \widehat{u}_\e^{\rm app}) \big) = \txtdiv \big( (\widehat{\mathcal{A}}_\e(x,x/\e_n) - \widehat{\mathcal{A}}^{\rm app}_\e(x,x/\e_n)) \nabla \widehat{u}_\e^{\rm app} \big)  \quad
\text{in } B_{5/4}, \\
& \widehat{u}_\e - \widehat{u}_\e^{\rm app} = 0 \quad \text{on } \partial B_{5/4},
\end{aligned} \right.
\end{equation}
and use \eqref{Aapp-A} to obtain
\begin{align}\label{u0v0-error}
\begin{split}
 \int_{B_{5/4}}|\nabla (\widehat{u}_\e - \widehat{u}_\e^{\rm app})|^2
&\le C L^2 \sum_{k=1}^{n-1}\e_{k}^{-2\tau} \e_{n} ^{2\tau}
 \int_{B_{5/4}} |\nabla \widehat{u}_\e^{\rm app}|^2 \\
 &\le C L^2 \sum_{k=1}^{n-1}\e_{k}^{-2\tau} \e_{n} ^{2\tau}
\bigg\{  \int_{B_2} |\na u_\e|^2  + \int_{B_2} |F|^2  \bigg\},
\end{split}
\end{align}
where the energy estimate for $\widehat{u}_\e^{\rm app}$ (or \eqref{2.29})  is also used in the last inequality. 

Finally, combining \eqref{ueve-error}, \eqref{chi.error} and \eqref{u0v0-error}, we may replace the approximate solutions $u_\e^{\rm app}, \widehat{u}_\e^{\rm app}, \chi_\e^{\rm app}$ on the left-hand side of \eqref{vev0-error-1} by $u_\e, \widehat{u}_\e, \chi_\e$, respectively. In particular,
by the property of the smoothing operator $S_{\e_n}$ (see \cite[Lemma 2.2]{nsxjfa2020}), we have
\begin{align} \label{seu0sev0-error}
\begin{split}
 &\int_{B_1}| S_{\e_n} \big((\na_y  \chi_\e )^{\e_n}  \na \widehat{u}_\e\big) -   S_{\e_n} \big( (\na_y  \chi_\e^{\rm app})^{\e_n} \na \widehat{u}_\e^{\rm app}\big)|^2\\
&\le 2\int_{B_1}\big\{| S_{\e_n} \big((\na_y  \chi_\e )^{\e_n}  (\na \widehat{u}_\e -\na \widehat{u}_\e^{\rm app} )\big)|^2 +| S_{\e_n} \big( ((\na_y  \chi_\e^{\rm app})^{\e_n}  -(\na_y  \chi_\e)^{\e_n} ) \na \widehat{u}_\e^{\rm app}\big)|^2 \big\}\\
&\le   C  \int_{B_{1+\e_n}}|  \na \widehat{u}_\e- \na \widehat{u}_\e^{\rm app}|^2  + C L^2 \sum_{k=1}^{n-1}\e_{k}^{-2\tau} \e_{n} ^{2\tau} \int_{B_{1+\e_n}}|  \na \widehat{u}_\e^{\rm app} |^2   \\
 &\le C L^2 \sum_{k=1}^{n-1}\e_{k}^{-2\tau} \e_{n} ^{2\tau}
\bigg\{  \int_{B_2} |\na u_\e|^2  + \int_{B_2} |F|^2  \bigg\}.
\end{split}
\end{align}
It follows that
\begin{equation}\label{ueu0-error}
\aligned
&  \int_{B_{5/4}}
|\na u_\e - \na \widehat{u}_\e- S_{\e_n}( (\na_y  \chi )^{\e} \na \widehat{u}_\e)|^2  \\
& \le C \big\{ \e_n^{1-\frac{2}{q}}
+  L^2\sum_{k=1}^{n-1}\e_{k}^{-2\tau} \e_n^{2\tau} \big\}
\bigg\{  \int_{B_{2}}
|\na u_\e|^2
+  \int_{B_{2}} |F|^2
\bigg\}.
\endaligned
\end{equation}
Recall that $\e_{n}$ is $Q$-separated from $\e_k, k=1,\cdots,n-1$, i.e., $\e_k/\e_{n}\geq Q$ for $k=1,\cdots, n-1.$ 
 In view of \eqref{mathcalX} and \eqref{formA},
the desired estimate \eqref{th-app-n-re1} follows immediately from \eqref{ueu0-error}, while the estimate \eqref{hatu.energy} follows directly from \eqref{u0v0-error}. 
The proof is complete.
\end{proof}

By similar arguments as above, it is not difficult to prove the boundary version of the uniform approximation theorem.

\begin{theorem}\label{th-app-bd}
     Let $\Omega$ be a bounded $C^1$ domain. Suppose that $A$ satisfies the assumptions \eqref{ellipticity}, \eqref{periodicity} and \eqref{smoothness-n}. Let $x_0\in \partial \Omega, B_r = B(x_0,r)$ with $0<r<r_0 = r_0(\Omega)$ and $u_\e$ be a weak solution to
     \begin{align}\label{thb-app-eq1}
     -\txtdiv \big( A(x/\e_1,x/\e_2,\cdots, x/\e_n) \na u_\e \big)=F  \, \text{ in } B_{2r}\cap \Omega, \quad   u_\e=0\,\text{ on } B_{2r}\cap \pa \Omega.
     \end{align}
     Then for any $Q>1$, there exist a coefficient matrix $ A^\flat = A^\flat(y_1,\cdots, y_{n-1})$, scales $\e_k'\in (0,\infty) $ with $ k=1,\cdots,n-1$, and a weak solution $u^\flat_{\e'}$  to
     $$
     -\txtdiv \big( A^\flat(x/\e'_1,x/\e'_2,\cdots, x/\e'_{n-1}) \na u^\flat_{\e'} \big)=F  \quad \text{in } B_r, \quad u^\flat_{\e'} = 0 \text{ on } B_r \cap \partial\Omega,
     $$
       such that 
       \begin{equation}
           \|\na u^\flat_{\e'} \|_{L^2(B_r \cap \Omega)}  \leq  C\big\{\| \na u_\e \|_{L^2(B_{2r}\cap \Omega)} +r\|F\|_{L^2(B_{2r} \cap \Omega)} \big\},
       \end{equation}
       and
    \begin{equation}\label{thb-app-re}
        \| \na u_\e - \na u^\flat_{\e'} -U_{\e'}\|_{L^2(B_r\cap \Omega)} \le C  \bigg\{ \bigg(\frac{Q^{n-1}  \e_n}{ r}\bigg) ^\sigma+ Q^{-\tau} \bigg\}  \big\{\| \na u_\e \|_{L^2(B_{2r}\cap \Omega)} +r\|F\|_{L^2(B_{2r} \cap \Omega)} \big\}
    \end{equation} for some $\sigma>0,$
    where $ U_{\e'} $ is given by
    \begin{equation}
        U_{\e'}(x) = (\e_n')^{- d}\int_{\mathbb{R}^d}
        \varphi  \bigg(\frac{x-z}{ \e_n'}\bigg)  (\na_y\mathcal{X}^\sharp) \bigg(\frac{z}{\e_1'},\cdots,\frac{z}{\e'_{n-1}}, \frac{x}{\e_n'}\bigg)  \na u^\flat_{\e'}(z)  dz,
    \end{equation}
    and $\mathcal{X}^\sharp$ is the corrector given by \eqref{cell-1} with $A$ replaced by $A^\sharp$ (we have extended $u^\flat_{\e'}(z)$ by zero across the boundary).
    Moreover, the constant $C$ depends only on $d, n, \Lm$ and  $(\tau,L)$ in \eqref{smoothness-n}, and the assumptions \eqref{ellipticity}, \eqref{periodicity} and \eqref{smoothness} are preserved for $A^\flat$. 
\end{theorem}

 \section{The uniform Calder\'{o}n-Zygmund estimates} \label{sec.4}
  
In this section we prove the uniform Calder\'{o}n-Zygmund  estimate for the equation \eqref{operaotr-n}. Let us first consider the interior estimate.

 \begin{theorem}\label{th-in-w1p}
 Assume $A$ satisfies the assumptions \eqref{ellipticity}, \eqref{periodicity} and \eqref{smoothness}. Let $u_\e$ be a weak solution to
 \begin{align} \label{th-in-w1p-eq}
 -\txtdiv \big( A(x/\e_1,x/\e_2,\cdots, x/\e_n) \na u_\e \big)= \txtdiv f\quad \text{in } 5B_0,
 \end{align} with $f\in L^p(5B_0)^d$ for some $2<p<\infty$, where $B_0=B(x_0,r_0)$ for some $x_0\in \R^d, r_0>0$. Then $\na u_\e \in L^p(B_0)^d$ and
 \begin{align} \label{th-in-w1p-re}
\bigg(  \fint_{B_0}     |\na u_\e|^p \bigg)^{1/p} \leq C \bigg\{ \bigg( \fint_{5B_0} |\na u_\e|^2\bigg)^{1/2}+\bigg(\fint_{5B_0 } |f|^p\bigg)^{1/p}  \bigg\},
\end{align}
where $C$ depends only on $d, \Lambda, p, n,$ and $(\tau,L)$ in \eqref{smoothness}.
 \end{theorem}
 The proof of  Theorem  \ref{th-in-w1p} relies on a  real-variable argument (a quantitative refined version of the one in \cite{caffarelli1998}) involving a double-averaging estimate. We define the averaging operator as
\begin{equation}
    M_t[f](x) :=\bigg( \fint_{B_t(x)} |f|^2 \bigg)^{1/2}.
\end{equation}
We need the averaging operator in the real-variable argument for two reasons. The first reason is that the approximation in \eqref{th-app-re} is meaningful only if $r \gg Q^{n-1} \e_n$ for sufficiently large $Q \gg 1$. This forces us to take the large-scale averaging estimate above the scale $t = CQ^{n-1} \e_n$ for sufficiently large $C$ and ignore the possible irregularity below that scale. The second reason to introduce another averaging is due to the special structure of the corrector $\mathcal{X}^{\sharp}$ in $U_{\e'}$ given by \eqref{Ue}. Recall that $\mathcal{X}^{\sharp} = \mathcal{X}^{\sharp}(y_1,y_2,\cdots, y_n)$ is 1-periodic in each $y_i$. But its gradient in $y_n$ does not have an $L^p$ estimate uniformly in $Q$ before we prove it in case of $n$ oscillating scales (otherwise we will run into a circular reasoning), which causes a big problem in the real-variable argument. To resolve this issue, we take the averaging for another time at scale $t = \e'_n = q\e_n$, a scale comparable to the size of the periodic cell (depending on $Q$ as $1<q\le Q^{n-1}$). This allows us to use only the periodic structure of $\mathcal{X}^\sharp$ and the $L^2$ energy estimate (independent of $Q$). Overall, the principle is that we have to make sure all the constants, except for $\eta$, in the real-variable argument (Theorem \ref{th W1P real} below) are independent of $Q$.

 \begin{theorem}\cite[Theorem 2.1 and Remark 2.4]{shenjga2023}\label{th W1P real}
  Let $p_0>2$, $\mathcal{F}\in L^2(4B_0)$ and $G\in L^p(4B_0)$ for some $2<p<p_0$, where $B_0$ is a ball in $\mathbb{R}^d$. Let $0 \leq t < c_0 \text{diam}(B_0)$ be a fixed
number.  Suppose that for each ball $B\subset 2B_0$ with the property that $t < \text{diam}(B) \leq  c_0 \text{diam}(B_0)$, there exists two measurable functions $\mathcal{F}_B$ and $\mathcal{R}_B$ on $2B$ such that $|\mathcal{F}|\leq |\mathcal{F}_B|+|\mathcal{R}_B|$ on $2B$, and
\begin{align}
&\bigg(\fint_{2B} |\mathcal{R}_B|^{p_0}\bigg)^{1/p_0}\leq C_1 \bigg\{ \bigg(\fint_{4B} |\mathcal{F}|^2\bigg)^{1/2} + \sup_{B\subset B' \subset 4B_0} \bigg(\fint_{B' } |G|^2\bigg)^{1/2}\bigg\},\label{W1p_real_cond1}\\
  &\bigg(\fint_{2B} |\mathcal{F}_B|^2\bigg)^{1/2} \leq C_2  \sup_{B\subset B' \subset 4B_0}\bigg(\fint_{B' } |G|^2\bigg)^{1/2}+ \eta \bigg(\fint_{4B}  |\mathcal{F}|^2\bigg)^{1/2} , \label{W1p_real_cond2}
\end{align}
where $C_1,C_2>0, 0<c_0<1$, and $\eta\geq 0$. Then there exists $\eta_0 >0$, depending only on $C_1, C_2, c_0, p, p_0,$
 with the property that if $0 \leq \eta < \eta_0$, then
 \begin{align} \label{th W1P real re1}
\bigg(  \fint_{B_0}   M_t[\mathcal{F}]^p \bigg)^{1/p} \leq C \bigg\{ \bigg(\fint_{4B_0} |\mathcal{F}|^2\bigg)^{1/2}+\bigg(\fint_{4B_0} |G|^p\bigg)^{1/p}\bigg\},
\end{align}
where $C$ depends only on $d, C_1,C_2, c_0, p$ and $ p_0$. If $t=0$, then \eqref{th W1P real re1} is replaced by
\begin{align}\label{th W1P real re2}
\bigg(\fint_{B_0} |\mathcal{F}|^p\bigg)^{1/p} \leq C \bigg\{ \bigg(\fint_{4B_0} |\mathcal{F}|^2\bigg)^{1/2}+\bigg(\fint_{4B_0} |G|^p\bigg)^{1/p}\bigg\}.
\end{align}
\end{theorem}

\begin{proof}[Proof of Theorem \ref{th-in-w1p}] Let us first consider the case $f=0$.
 We use an inductive argument on the number of scales to prove the theorem. Note that the result is well-known for $n=1$ (see e.g., \cite{shenbook2018}).
Assume the theorem is true if the number of scales is strictly less than $n$, which means that for any $\infty >  \e_1 \ge \e_2 \ge \cdots \ge \e_{n-1} > 0$, the solution $v_\e$ to
$$-\txtdiv \big( A(x/\e_1, \cdots, x/\e_{n-1}) \na v_\e \big)=0 \quad \text{in }  B_{2r} $$
satisfies the uniform interior $W^{1,p}$ estimate
\begin{align}\label{p-thw1p-1}
\bigg(\fint_{B_{3r/2}} |\na v_\e |^p\bigg)^{1/p} \leq C_p  \bigg(\fint_{B_{2r}} |\na v_\e|^2\bigg)^{1/2} 
\end{align} for any $0<r<\infty$ and any $2<p<\infty$. 

We next prove the theorem for $n$ scales by using Theorem \ref{th W1P real} with $\mathcal{F}:=  M_{\e_n'}[\na u_\e] $ with $\e_n' = q\e_n$ given in \eqref{ei'}. Fix any $p\in (2,\infty)$ and $p_1\in (2, p)$.

Let $Q>1$ to be determined later, and $\e_n'=q\e_n\leq Q^{n-1} \e_n<r<c_0 r_0$ for some fixed $c_0 \in (0,1)$. For each ball $B_r=B_r(x)$ such that  $B_{r} \subset   2B_0 = B(x_0, 2r_0)$, we need to construct $\mathcal{F}_B$ and $\mathcal{R}_B$, which satisfy  \eqref{W1p_real_cond1} and \eqref{W1p_real_cond2}, respectively.
Since $u_\e$ is a solution to
 $
  -\txtdiv \big( A(x/\e_1, \cdots, x/\e_n) \na u_\e \big)=0$   in $B_{4r} \subset 4B_0$,
  by Theorem \ref{th-app}, there exist a 1-periodic matrix $A^\flat = A^\flat(y_1,y_2,\cdots,y_{n-1})$, scales $\e' = (\e_1',\cdots, \e_{n-1}')$, and a solution $  u_{\e'}^\flat(x) $ to
     \begin{align} \label{p-thw1p-2}
     -\txtdiv \big( A^\flat(x/\e'_1, \cdots,   x/\e'_{n-1}) \na u_{\e'}^\flat \big)=0  \quad \text{in } B_{2r},
     \end{align}
     such that
     \begin{align}\label{p-thw1p-3}
     \| \na u_{\e'}^\flat \|_{L^2(B_{2r})}\leq C  \| \na u_\e \|_{L^2(B_{4r})},
     \end{align} and
   \begin{align} \label{p-thw1p-4}
   \begin{split}
        &\| \na u_{\e} - \na u_{\e'}^\flat -U_{\e'}\|_{L^2(B_{2r})}
         \le C  \bigg\{  Q^{-\tau}+\bigg(\frac{Q^{n-1}  \e_n}{ r}\bigg) ^\sigma \bigg\}   \| \na u_\e \|_{L^2(B_{4r})},
                  \end{split}
    \end{align}
          where $U_{\e'}$ is given by \eqref{Ue}.
Let    
 \begin{align*}
\mathcal{F}_B(x)=  M_{\e_n'}[\na u_\e- \na u_{\e'}^\flat -U_{\e'}](x)  \,\,\text{ and }\,\, \mathcal{R}_B(x)= M_{\e_n'}[\na u_{\e'}^\flat + U_{\e'}](x) .
\end{align*} 
By the triangle inequality, we have $|\mathcal{F} | \le |\mathcal{F}_B |+ |\mathcal{R}_B|.$ 

We recall a basic property of the averaging operator $M_t$. If $F\in L^2_{\rm loc}(\R^d)$, then for any $r \ge t>0$, we have
\begin{equation}\label{p-thw1p-6}
    \fint_{B_r} M_t[F]^2 \le C\fint_{B_{r+t}} |F|^2 \le C\fint_{B_{2r}} |F|^2 \le C\fint_{B_{2r}} M_t[F]^2.
\end{equation}
This can be shown easily by the Fubini's theorem.

Consequently, by the assumption $r> \e_n'$, we have  
\begin{equation}
    \fint_{B_r }|\mathcal{F}_B|^2 \le C \fint_{B_{2r} } | \na u_{\e} - \na u_{\e'}^\flat -U_{\e'} |^2 ,
\end{equation}
and
\begin{equation}
    \fint_{B_{4r} } | \na u_{\e} |^2 \le C \fint_{B_{4r} }|\mathcal{F}|^2.
\end{equation}
Combining the last two inequalities with \eqref{p-thw1p-4}, we have
 \begin{align} \label{p-thw1p-7}
    \fint_{B_r }|\mathcal{F}_B|^2 \leq   C  \bigg\{  Q^{-\tau}+\bigg(\frac{Q^{n-1}  \e_n}{ r}\bigg) ^\sigma \bigg\}^2   \fint_{B_{4r} }|\mathcal{F}|^2.  
            \end{align}
We now choose $Q = Q_0>1$ large enough such that $CQ_0^{-\tau}< \eta_0/2$, and then take  $r> k_0 Q_0^{n-1} \e_n $ for some  $k_0$ large enough such that $C Q_0^{(n-1)\sigma} (\e_n/r)^\sigma   < \eta_0/2,$ where $\eta_0$ is given by Theorem \ref{th W1P real} depending only on $C_1, C_2, c_0, p, p_1.$
It follows that
$\mathcal{F}_B$ satisfies condition \eqref{W1p_real_cond2} for $r>k_0 Q_0^{n-1} \e_n$ with $C_2=0$. In the following, we will verify that $\mathcal{R}_B$ satisfies \eqref{W1p_real_cond1} and the constant $C_1$ is independent of $Q$.

Note that  
 \begin{align} \label{p-thw1p-8}
   \begin{split}
\fint_{B_r }|\mathcal{R}_B|^p& = \fint_{B_r } M_{\e_n'}[\na u_{\e'}^\flat + U_{\e'}]^p \\
&\leq  2^{p-1} \fint_{B_r }     |M_{\e_n'} [ \na u_{\e'}^\flat ] (z)| ^{p } dz    + 2^{p-1} \fint_{B_r }  |M_{\e_n'} [  U_{\e'} ](z) |^{p}   dz .
\end{split}
    \end{align}
Since $ u_{\e'}^\flat$ satisfies \eqref{p-thw1p-2}, and the coefficient matrix $ A^\flat$  involves at most $n-1$ scales and is periodic and  H\"{o}lder continuous $y_i=x/\e_i'$ for $i=1,\cdots,n-1$. By the inductive assumption, we know that $ u_{\e'}^\flat $ satisfies  \eqref{p-thw1p-1},
 which combined with \eqref{p-thw1p-3} implies
\begin{align*} 
\bigg(\fint_{B_{3r/2}} | \na u_{\e'}^\flat |^p\bigg)^{1/p} \leq C  \bigg(\fint_{B_{4r}} | \na u_\e  |^2\bigg)^{1/2},
\end{align*} 
for any $ k_0\e_n'\le k_0 Q_0^{n-1 } \e_n \leq r \leq c_0 r_0 $. 
By the H\"{o}lder's inequality and Fubini's theorem,  
    \begin{align}\label{p-thw1p-9}
    \begin{split}
    \fint_{B_r }     |M_{\e_n'} [ \na u_{\e'}^\flat ] (z)| ^{p } dz  
    & \leq \fint_{B_r } \fint_{B_{\e_n'}(z)}  |  \na u_{\e'}^\flat (y)|^p  dy dz\\
    & \leq  C \fint_{B_{3r/2} }   |  \na u_{\e'}^\flat (y)|^p  dy\\
    &\leq C \bigg(\fint_{B_{4r} } | \na u_\e  |^2\bigg)^{p/2}\\
    &\leq C \bigg(\fint_{B_{4r} }|\mathcal{F}|^2\bigg)^{p/2}, 
    \end{split}
    \end{align}
where we have used \eqref{p-thw1p-6} in the last step.

To bound the second term on the right hand side of \eqref{p-thw1p-8}, we have to take advantage of the periodic structure of $\mathcal{X}^\sharp$ and it is right here that the averaging operator $M_{\e_n'}$ plays a key role.
We observe that, for each $y\in B_r$,
\begin{align*} 
\begin{split}
   |M_{\e_n'} [  U_{\e'} ](y) |^2 &=  \fint_{B_{\e_n'}(y)} \bigg|(\e_n')^{- d}\int_{\mathbb{R}^d}
   \varphi  \bigg(\frac{x-z}{ \e_n'}\bigg)  (\na_y\mathcal{X}^\sharp ) \bigg(\frac{z}{\e_1'},\cdots,\frac{z}{\e'_{n-1}}, \frac{x}{\e_n'}\bigg)  \na u_{\e'}^\flat(z)   dz \bigg|^2 dx \\
   &\leq \fint_{B_{\e_n'}(y)}(\e_n')^{- d}\int_{B_{2\e_n'}(y)}
\varphi  \bigg(\frac{x-z}{ \e_n'}\bigg) \bigg| (\na_y\mathcal{X}^\sharp) \bigg(\frac{z}{\e_1'},\cdots,\frac{z}{\e'_{n-1}}, \frac{x}{\e_n'}\bigg)  \na u_{\e'}^\flat(z)\bigg|^2  dz dx\\
&\leq \fint_{B_{2\e_n'}(y)}(\e_n')^{- d}\int_{B_{\e_n'}(y)}
\varphi  \bigg(\frac{x-z}{ \e_n'}\bigg) \bigg| (\na_y\mathcal{X}^\sharp) \bigg(\frac{z}{\e_1'},\cdots,\frac{z}{\e'_{n-1}}, \frac{x}{\e_n'}\bigg) \bigg|^2 dx |\na u_{\e'}^\flat(z)|^2  dz \\
   &\leq C \sup_{z\in \R^d}\fint_{\T^d}  \bigg|(\na_y\mathcal{X}^\sharp) \bigg(\frac{z}{\e_1'},\cdots,\frac{z}{\e'_{n-1}}, y\bigg)\bigg|^2 dy  \fint_{B_{2\e_n' }(y)}   |\na u_{\e'}^\flat(z)|^2  dz\\
   & \leq  C \fint_{B_{2\e_n' }(y)}   |\na u_{\e'}^\flat(z)|^2dz = C |M_{2\e_n'} [\nabla u_{\e'}^\flat](y)|^2,
    \end{split}
\end{align*} 
where we have used Cauchy-Schwarz inequality in the first inequality (splitting a half of $(\e_n')^{-d}\varphi(\frac{x-z}{\e_n'})$ whose $L^2$ integral is 1), Fubini's theorem in the second inequality, boundedness of $\varphi$ and 1-periodicity of $\mathcal{X}^\sharp$ in the third inequality (as well as H\"{o}lder's inequality in $z$ variable), and the $L^2$ energy estimate for $\mathcal{X}^\sharp$ in a periodic cell (see the equation \eqref{cell-n}) in the last inequality. The constant $C$ in the above inequality depends only on $d $ and $\Lm$. 
 This combined with \eqref{p-thw1p-9} implies that
\begin{align}\label{p-thw1p-10}
\begin{aligned}
 \fint_{B_r}  |M_{\e_n'} [  U_{\e'} ](z) |^p    
    & \leq  C   \fint_{B_r}    |M_{2\e_n'} [ \na u_{\e'}^\flat ] (y)| ^{p } dy \\
    &\leq C\fint_{B_{4r/3}}    |M_{\e_n'} [ \na u_{\e'}^\flat ] (y)| ^{p } dy
    \leq C \bigg(\fint_{B_{4r}}|\mathcal{F}|^2\bigg)^{p/2},
    \end{aligned}
\end{align} 
for any $k_0 Q_0^{n-1} \e_n \leq r < c_0 r_0$. We point out that we cannot simply take the $L^p$ norm of $U_{\e'}$ directly since $\| \nabla \mathcal{X}^\sharp \|_{L^p(B_1, dy_n)}$ depends on $q$ and hence on $Q$, due to the loss of regularity of  $A^\sharp(y_1,y_2,\cdots, y_n)$ in $y_n$ caused by reperiodization.

 By  taking  \eqref{p-thw1p-9} and  \eqref{p-thw1p-10} into \eqref{p-thw1p-8}, we find that $\mathcal{R}_B  $ satisfies \eqref{W1p_real_cond1} for $ k_0 Q_0^{n-1} \e_n <r < c_0 r_0$ (without the $G$ part) with $C_1$ independent of $Q$. Thus the conditions of Theorem \ref{th W1P real} are all satisfied.
Thanks to \eqref{th W1P real re1}, we obtain that
\begin{align} \label{p-thw1p-11}
\bigg(  \fint_{B_0}    M_{\e_n''} [ M_{\e_n'}[\na u_\e]]^{p_1} \bigg)^{1/{p_1}} \leq C  \bigg(\fint_{4B_0} |M_{\e_n'}[\na u_\e](x) |^2\bigg)^{1/2},
\end{align}
where $\e_n' = q\e_n$ and $\e_n'' = k_0 Q_0^{n-1} \e_n$ and $p_1\in (2,p)$. This is a double-averaging estimate since we have taken average twice in the $L^p$ norm on the left-hand side. At this stage (after using Theorem \ref{th W1P real}), it is safe to remove one averaging in \eqref{p-thw1p-11} by \eqref{p-thw1p-6}. In fact, since $\e_n'' > \e_n'$, \eqref{p-thw1p-11} implies
\begin{equation}\label{en''Scale-CZ}
    \bigg(  \fint_{B_0}    M_{2\e_n''} [\na u_\e]^{p_1} \bigg)^{1/{p_1}} \leq C  \bigg(\fint_{5B_0} | \na u_\e|^2\bigg)^{1/2}.
\end{equation}

Next, we claim that for each $y\in B_0$, we have
\begin{equation}
    |\nabla u_\e(y)| \le C M_{2\e_n''} [\na u_\e](y) = C\bigg( \fint_{B_{2\e_n''}(y)} |\na u_\e|^2 \bigg)^{1/2}.
\end{equation}
This pointwise Lipschitz estimate follows from a well-known blow-up argument at small scales and the fact $\e_n'' = k_0 Q_0^{n-1} \e_n$. Inserting this into \eqref{en''Scale-CZ}, we arrive at
\begin{equation}
    \bigg(  \fint_{B_0}    |\na u_\e|^{p_1} \bigg)^{1/p_1} \leq C  \bigg(\fint_{5B_0} | \na u_\e|^2\bigg)^{1/2}.
\end{equation}
Since $p_1\in (2,p)$ and $p\in (2,\infty)$ are arbitrary with $C$ depending on $p,p_1$, we actually have shown the uniform interior $W^{1,p}$ estimate for any $p\in (2,\infty)$ in the case of $f = 0$.

Finally, we consider general case with  $f\in L^p(5B_0)^d$.  Let $u_\e$ be a solution to \eqref{th-in-w1p-eq}.
 For each ball $B=B_r$ such that  $4 B \subset 2B_0$, let $u_{\e,1} \in H^1_0(4B)$ be the solution  to
 \begin{align} \label{p-thw1p-13}
  -\txtdiv \big( A(x/\e_1,x/\e_2,\cdots, x/\e_n) \na u_{\e,1} \big)= \text{div } f\,\text{ in }  4B ,
  \end{align}
  and $u_{\e,2}$ the solution to
  \begin{align} \label{p-thw1p-14}
  -\txtdiv \big( A(x/\e_1,x/\e_2,\cdots, x/\e_n) \na u_{\e,2} \big)=0\,\text{ in }  4B ,\, \text{ and } \,u_{\e,2}=u_\e \,\text{ on }  \pa (4B) .
  \end{align}
  Thus, $u_\e = u_{\e,1} + u_{\e,2}$ in $4B$.
The energy estimate for \eqref{p-thw1p-13} implies that
$$\bigg(\fint_{4B} |\na u_{\e,1}|^2\bigg)^{1/2} \leq    C \bigg(\fint_{4B}  |f|^2\bigg)^{1/2},$$
while the uniform $W^{1,p}$ estimates for $u_{\e,2}$ proved above implies that
\begin{align*}
\bigg(\fint_{2B} |\na u_{\e,2}|^{p_0}\bigg)^{1/p_0} &\leq  C  \bigg(\fint_{4B} |\na u_{\e,2}|^2\bigg)^{1/2}\\
&\leq C \bigg(\fint_{4B} |\na u_\e|^2\bigg)^{1/2} +   C \bigg(\fint_{4B}  |f|^2\bigg)^{1/2}
\end{align*}
for some $p_0 \in (p, \infty)$.
Let $ \mathcal{F}=|\na u_\e|$,  $\mathcal{F}_B= |\na u_{\e,1}|$, and $\mathcal{R}_B= |\na u_{\e,2} |$. We obtain \eqref{th-in-w1p-re}  from Theorem \ref{th W1P real} immediately (with $t = 0$).
\end{proof}

Theorem \ref{th-in-w1p} provides the uniform interior $W^{1,p}$  estimate for the equation \eqref{operaotr-n}. 
Based on Theorem \ref{th-app-bd} and a boundary version of Theorem \ref{th W1P real} (See Theorem 4.1 and Remark 4.2 in \cite{shenjga2023}.), we can follow the same arguments to prove the following uniform boundary $W^{1,p}$ estimate.
 \begin{theorem}\label{th-b-w1p}
 Let $\Omega$ be a bounded $C^1$ domain. There exists $r_0 = r_0(\Omega)$ such that the following statement is true. Let $B=B(x,r) $ with $x \in \pa\Omega$, $0<r< r_0$.   Assume $A$ satisfies the assumption \eqref{ellipticity}, \eqref{periodicity} and \eqref{smoothness}.  Let $u_\e$ be a weak solution to
 \begin{align} \label{th-b-w1p-eq}
 -\txtdiv \big( A(x/\e_1, \cdots, x/\e_n) \na u_\e \big)= \txtdiv f \, \text{ in } B_{2r}\cap \Omega, \quad   u_\e=0 \,\text{ on } B_{2r}\cap \pa \Omega,
 \end{align}
 with $f\in L^p(B_{2r}\cap \Omega)^d$ for some $2<p<\infty$. Then $\na u_\e \in L^p(B_{r}\cap \Omega)^d$ and
 \begin{align} \label{th-b-w1p-re}
\bigg(  \fint_{B_{r}\cap \Omega}     |\na u_\e|^p \bigg)^{1/p} \leq C \bigg\{ \bigg( \fint_{B_{2r}\cap \Omega} |\na u_\e|^2\bigg)^{1/2}+\bigg(\fint_{B_{2r}\cap \Omega } |f|^p\bigg)^{1/p}  \bigg\},
\end{align}
where $C$ depends only on $d, \Lambda, p, n, (\tau,L)$ in \eqref{smoothness}, and the $C^1$ character of $\Omega$.
 \end{theorem}

We finally provide the proof of our main results.
\begin{proof} [Proof of Theorem \ref{th-w1p}]
Note that the case $p=2$ follows from standard energy estimates, while the case $p>2$ is a direct consequence of Theorems \ref{th-in-w1p} and \ref{th-b-w1p} together with a covering argument.  Finally, the case $1<p<2$ follows from the case $p>2$ and a standard duality argument.
\end{proof}

\section{Large-scale Lipschitz estimates}\label{sec.5}
In this section, we investigate the large-scale Lipschitz estimates for  \eqref{operaotr-n}.
We distinguish between the cases of two scales and more scales. We recall that under the scale-separation condition: there exists a positive integer $N$  such that
\begin{equation}\label{w-s-cond}
\left( \frac{\e_{i+1}}{\e_i} \right)^N \le   \frac{\e_i}{\e_{i-1}}    \quad \text{ for } 1\le i\le n-1,
\end{equation} 
the large-scale Lipschitz estimate for the multiscale elliptic operators has been derived in \cite{nsxjfa2020}. 
Indeed, let $u_\e$ be the weak solution to 
 $$-\txtdiv (A(x,x/\e_1,\cdots,x/\e_n)  \nabla u_\e) = F \quad \text{ in } B_1$$ with $F\in L^p(B_1)$ for some $p>d$.
 Suppose  $A(x,y_1,\cdots,y_n)$ is strongly elliptic, periodic in  $y_i$ for $ 1\le i \le n$, and
H\"{o}lder continuous in $x, y_1,\cdots,y_{n-1}$ (no smoothness is needed for $y_n$). Then, under the condition \eqref{w-s-cond}, for any $0<\e_n \leq r <1$, 
\begin{align}\label{relipth}
\bigg(\fint_{B_r} |\na u_\e|^2\bigg)^{1/2}\leq C \bigg\{  \bigg(\fint_{B_1}
|  \na u_\e|^2\bigg)^{1/2} +
 \bigg(\fint_{B_1} |F  |^p\bigg)^{1/p}\bigg\},
\end{align}
where $C$ is independent of $\e$.

However, as we have stated before, without any scale separation condition the full-scale uniform Lipschitz estimate  for \eqref{operaotr-n} seems difficult and still remains open. The following theorems provide suboptimal results in this direction using the idea of scale separation similar to Section \ref{sec3.1}.

\begin{theorem}\label{thm.n=2}
    Assume $A(y_1,y_2)$ satisfies \eqref{ellipticity}, \eqref{periodicity} and \eqref{smoothness-n} with $n=2$.  For $B_1=B_1(0)$, let $u_\e$ be a weak solution of $-\txtdiv (A_\e \nabla u_\e) = F$ in $B_1$ with $F\in L^p(B_1)$ for some $p>d$. Then for any $\alpha>0$, and any $\e_2^{1-\alpha} \le r\le 1$, we have
    \begin{equation}
        \bigg( \fint_{B_r} |\nabla u_\e|^2 \bigg)^{1/2} \le C_\alpha \bigg\{ \bigg( \fint_{B_1} |\nabla u_\e|^2 \bigg)^{1/2} + \bigg( \fint_{B_1} |F|^p \bigg)^{1/p} \bigg\},
    \end{equation}
    where $C_\alpha$ depends only on $d,\Lambda,p, \alpha,$ and $(\tau,L)$ in \eqref{smoothness-n}.
\end{theorem}

\begin{proof}
    Applying the argument in Section \ref{sec3.1} and by \eqref{eq.rewrite}, for any $Q>1$ we can find a 1-periodic matrix $A^\sharp(y_1,y_2)$, H\"{o}lder continuous in $y_1$, and write 
    \begin{equation}
        A\big(\frac{x}{\e_1}, \frac{x}{\e_2}\big) = A^\sharp\big(\frac{\gamma_1 x}{\e_2}, \frac{x}{q \e_2}\big) = A^\sharp\big(\frac{x}{\e_1'}, \frac{x}{\e_2'}\big),
    \end{equation}
    where $1\le q\le Q, \e_1' = \e_2/\gamma_1$ and $\e_2' = q\e_2$. Let $Q = \e_2^{-\alpha}$. Then $\e_2 \le \e_2' = q\e_2 \le \e_2^{1-\alpha}$ and $\e_1'/\e_2' \ge Q = \e_2^{-\alpha}$. Thus $\e_2'$ and $\e_1'$ are well-separated. Moreover, $u_\e$ satisfies the equation
    \begin{equation}
        -\txtdiv (A^\sharp(x/\e_1', x/\e_2') \nabla u_\e) = F  \qquad \text{in } B_1 .
    \end{equation}
   Thanks to \eqref{relipth}, we have for $\e_2' \le r\le 1$,
    \begin{equation}
        \bigg( \fint_{B_r} |\nabla u_\e|^2 \bigg)^{1/2} \le C_\alpha \bigg\{ \bigg( \fint_{B_1} |\nabla u_\e|^2 \bigg)^{1/2} + \bigg( \fint_{B_1} |F|^p \bigg)^{1/p} \bigg\}.
    \end{equation}
    This implies the desired estimate since $\e_2'\le \e_2^{1-\alpha}$.
    \end{proof}

\begin{theorem}\label{thm.n>2} 
Assume $A(y_1,\cdots,y_n)$ satisfies \eqref{ellipticity}, \eqref{periodicity} and \eqref{smoothness} with  
     $n \ge 3$. Let $u_\e$ be a weak solution of $-\txtdiv (A_\e \nabla u_\e) = F$ in $B_1$ with $F\in L^p(B_1)$ for some $p>d$. Then for any $\alpha \in (0,1/n]$ and $\delta \in [\e_n,1)$, there exists $1\le q \le \delta^{-n\alpha}$, such that for any $q\delta \le r\le r_1 := \min\{ q\delta^{1-\alpha}, 1\}$, we have
    \begin{equation}
        \bigg( \fint_{B_r} |\nabla u_\e|^2 \bigg)^{1/2} \le C_\alpha \bigg\{ \bigg( \fint_{B_{r_1}} |\nabla u_\e|^2 \bigg)^{1/2} + r_1^{1-d/p}\bigg( \fint_{B_{r_1}} |F|^p \bigg)^{1/p} \bigg\},
    \end{equation}
    where $C_\alpha$ depends only on $d,n,\Lambda,p,\alpha$, and $(\tau,L)$ in \eqref{smoothness}. Moreover, if $\delta = \e_j$ for some $1\le j\le n$, then $\alpha$ can be taken from $(0,1/(n-1)]$ and $1\le q\le \delta^{-(n-1)\alpha}$.
\end{theorem}
\begin{proof}
    Fix $\delta \in [\e_n, 1)$. Let $Q = \delta^{-\alpha}>1$ with $\alpha\in (0,1/n]$. By the Dirichlet's theorem, there exist $q,p_1,p_2,\cdots, p_n$ such that $1<q\le Q^n = \delta^{-n\alpha}$ and
    \begin{equation}\label{DThm.delta}
       \big |\frac{\delta}{\e_i} - \frac{p_i}{q}\big| \le \frac{1}{qQ},
    \end{equation}
    where $i=1,2,\cdots, n$. As in Section \ref{sec3.1}, we set
    \begin{equation*}
        \gamma_i: = \Big| \frac{\delta}{\e_i} - \frac{p_i}{q} \Big| \quad \text{and} \quad s_i = \sgn\Big( \frac{\delta}{\e_i} - \frac{p_i}{q} \Big),
    \end{equation*}
    and write
    \begin{equation*}
        \frac{1}{\e_i} = \frac{\gamma_i s_i}{\delta} + \frac{p_i}{q} \frac{1}{\delta}.
    \end{equation*}
    Define
    \begin{equation}
        A_1^\sharp(y_1,y_2,\cdots, y_n, y_{n+1}) = A(s_1 y_1+p_1 y_{n+1}, \cdots, s_n y_n + p_n y_{n+1}).
    \end{equation}
    As a consequence, we have
    \begin{equation}
        A\big(\frac{x}{\e_1}, \cdots, \frac{x}{\e_n}\big) = A_1^\sharp\big(\frac{\gamma_1 x}{\delta}, \cdots, \frac{\gamma_n x}{\delta}, \frac{x}{q\delta}\big).
    \end{equation}
    The last identity shows that we can rewrite original coefficient matrix $A$ with $n$ scales as a new 1-periodic matrix $A_1^\sharp$ with $n+1$ scales. Moreover, the smallest scale $q\delta$ is at least $Q$-separated from the remaining $n$ scales, i.e.,
    \begin{equation}
        \frac{\delta/\gamma_i}{q\delta} \ge Q.
    \end{equation}
    Hence, by a blow-up argument and \eqref{relipth} we have 
    \begin{equation}
        \bigg( \fint_{B_r} |\nabla u_\e|^2 \bigg)^{1/2} \le C_\alpha \bigg\{ \bigg( \fint_{B_{r_1}} |\nabla u_\e|^2 \bigg)^{1/2} + r_1^{1-d/p}\bigg( \fint_{B_{r_1}} |F|^p \bigg)^{1/p} \bigg\}  
    \end{equation}for any $r$ with $q\delta \le r\le r_1 \le \min \{ Qq\delta,1\}$. 
    This is the desired estimate. To make sure $q\delta <1$, we require $q\delta \le Q^n \delta = \delta^{1-n\alpha} < 1$ which gives $\alpha \le 1/n$.

    Finally, in the particular case $\delta = \e_j$ for some $1\le j\le n$, $\delta/\e_j = 1$ leads to a trivial approximation in \eqref{DThm.delta}. Thus we will only need \eqref{DThm.delta} for $i\neq j$ and $1\le q \le Q^{n-1}$. The rest of the proof follows similarly.
\end{proof}

\bibliography{Ref}
\bibliographystyle{alpha}
\medskip
\end{document}